\documentclass{amsart}

\usepackage{euscript, rotating, amsthm, amssymb}

\newtheorem{theorem}{Theorem}[section]
\newtheorem{lemma}[theorem]{Lemma}
\newtheorem{remark}[theorem]{Remark}
\newtheorem{proposition}[theorem]{Proposition}
\newtheorem{corollary}[theorem]{Corollary}

\newtheorem{definition}[theorem]{Definition}

\newcommand{\PenkovPetukhovBr}[1]{\mathrm{span}\{#1\}}
\title{On ideals in $\operatorname{U}(\mathfrak{sl}(\infty))$, $\operatorname{U}(\mathfrak o(\infty)),$ $\operatorname{U}(\mathfrak{sp}(\infty))$}
\author{Ivan Penkov, Alexey Petukhov}

\begin{document}
\begin{abstract} We provide a review of results on two-sided ideals in the enveloping algebra $\operatorname{U}(\mathfrak g(\infty))$ of a locally simple Lie algebra $\mathfrak g(\infty)$. We pay special attention to the case when $\mathfrak g(\infty)$ is one of the finitary Lie algebras $\mathfrak{sl}(\infty), \mathfrak o(\infty), \mathfrak{sp}(\infty)$. The main results include a description of all integrable ideals in $\operatorname{U}(\mathfrak g(\infty))$, as well as a criterion for the annihilator of an arbitrary (not necessarily integrable) simple highest weight module to be nonzero. This criterion is new for $\mathfrak g(\infty)=\mathfrak o(\infty), \mathfrak{sp}(\infty)$. All annihilators of simple highest weight modules are integrable ideals for $\mathfrak g(\infty)=\mathfrak{sl}(\infty),$ $\mathfrak o(\infty)$. Finally, we prove that the lattices of ideals in $\operatorname{U}(\mathfrak o(\infty))$ and $\operatorname{U}(\mathfrak{sp}(\infty))$ are isomorphic.
\end{abstract}


\maketitle

\section{Introduction and outline of results}
The purpose of this paper is to provide a review of results on two-sided ideals in the enveloping algebra $\operatorname{U}(\mathfrak g(\infty))$ of an infinite-dimensional Lie algebra $\mathfrak g(\infty)$
obtained as the inductive limit of an arbitrary
chain of embeddings of simple finite-dimensional Lie
algebras\begin{equation}\mathfrak g(1)\hookrightarrow\mathfrak g(2)\hookrightarrow...\hookrightarrow\mathfrak
g(n)\hookrightarrow...\label{Ech1}\end{equation}with $\lim\limits_{n\to\infty}\dim\mathfrak
g(n)=\infty$. We mostly focus on the simple finitary complex Lie algebras $$\mathfrak g(\infty)=\mathfrak{sl}(\infty), \mathfrak o(\infty), \mathfrak{sp}(\infty)$$
for which we establish some new results, so this article is a combination of a review and a research article.

A simplest motivation for the study of the Lie algebras $\mathfrak{sl}(\infty), \mathfrak o(\infty), \mathfrak{sp}(\infty)$ and their representations is the necessity to study stabilization effects in the representation theory of the classical simple finite-dimensional Lie algebras $\mathfrak{sl}(n), \mathfrak o(n), \mathfrak{sp}(2n)$ when $n\to\infty$. At a deeper level, the challenge is to develop a representation theory of $\mathfrak g(\infty)=\mathfrak{sl}(\infty), \mathfrak o(\infty), \mathfrak{sp}(\infty)$ which 
does not refer to $n$ as in $\mathfrak{sl}(n)$, $\mathfrak o(n)$ or $\mathfrak{sp}(2n)$.

There have been some first successes in this direction, for example the discovery of the category of tensor modules $\mathbb T_{\mathfrak g(\infty)}$~\cite{DPS}; this category can be considered as ``the common core'' of the categories of finite-dimensional representations of all classical finite-dimensional Lie algebras of given type $\mathfrak{sl}, \mathfrak o$ or $\mathfrak{sp}$.

The study of ideals in $\operatorname{U}(\mathfrak g(\infty))$, especially primitive ideals, is another topic in which there are interesting results. The reason for studying primitive ideals is clear: Dixmier's observation that classifying primitive ideals in $\operatorname{U}(\mathfrak g)$ is a potentially manageable task while classifying all irreducible representations of a Lie algebra $\mathfrak g$ is unrealistic, applies with full strength to the case of $\mathfrak g(\infty)=\mathfrak{sl}(\infty), \mathfrak o(\infty), \mathfrak{sp}(\infty)$. Despite the fact that we do not have yet the classification of primitive ideals of $\operatorname{U}(\mathfrak g(\infty))$, we hope that we are close to such a classification, and that it will be useful to have a single source for the results achieved so far\footnote{Recently we have shown that all primitive ideals of $\operatorname{U}(\frak{sl}(\infty))$ are integrable~\cite{PP3}. This, together with Proposition~\ref{Pclsm} of the present paper, yields a classification of primitive ideals of $\operatorname{U}(\frak{sl}(\infty))$. As a consequence, Problem c) below is now also answered in the affirmative for $\frak{sl}(\infty)$ via Theorem~5.4.}.

The main effect which distinguishes the case of $\operatorname{U}(\mathfrak g(\infty))$ from the case of $\operatorname{U}(\mathfrak g)$ for a finite-dimensional Lie algebra $\mathfrak g$ is that $\operatorname{U}(\mathfrak g(\infty))$ has ``fewer'' ideals than $\operatorname{U}(\mathfrak g)$: we conjecture that $\operatorname{U}(\mathfrak g(\infty))$ has only countably many ideals. This latter statement is partially supported by the fact that the annihilator in $\operatorname{U}(\mathfrak g(\infty))$ of a generic simple highest weight $\mathfrak{g}(\infty)$-module equals zero.

We now describe the contents of the paper. The ground field is algebraically closed of characteristic 0. We start with results concerning the associated pro-variety of a proper two-sided ideal $I$ in $\operatorname{U}(\mathfrak g(\infty))$ for an arbitrary locally simple Lie algebra $\mathfrak g(\infty)=\varinjlim\mathfrak g(n)$.  It turns out that $I$ can have a proper associated pro-variety (i.e. an associated pro-variety different from 0 or from the coadjoint representation $\mathfrak g(\infty)^*:=\varprojlim \mathfrak g(n)^*$) only when $\mathfrak g(\infty)$ is finitary, i.e. is isomorphic to one of the three  infinite-dimensional Lie algebras $\mathfrak{sl}(\infty)$, $\mathfrak o(\infty)$, $\mathfrak{sp}(\infty)$, see A.~Baranov's classification of simple finitary Lie algebras~\cite{Ba2}. This is one of the main results of our paper~\cite{PP1} and we do not reproduce the proof here. This result leads relatively quickly to a proof of Baranov's conjecture that $\operatorname{U}(\mathfrak g(\infty))$ admits proper two-sided ideals different from the augmentation ideals if and only if $\mathfrak g(\infty)$ is diagonal.

Diagonal locally simple Lie algebras are a very interesting generalization of the three finitary Lie algebras $\mathfrak{sl}(\infty), \mathfrak o(\infty), \mathfrak{sp}(\infty)$: a classification of diagonal locally simple Lie algebras has been given by A.~Baranov and A.~Zhilinskii~\cite{BZh1}. For a diagonal Lie algebra $\mathfrak g(\infty)$, nonisomorphic to $\mathfrak{sl}(\infty), \mathfrak o(\infty), \mathfrak{sp}(\infty)$, a classification of two-sided ideals of $\operatorname{U}(\mathfrak g(\infty))$ follows from the work of A.~Zhilinskii~\cite{Zh3}, see Section~\ref{Sclsint}.

The case $\mathfrak g(\infty)=\mathfrak{sl}(\infty), \mathfrak{o}(\infty), \mathfrak{sp}(\infty)$ is the most interesting case and it plays a distinguished role in this review. Let $\mathfrak g(\infty)=\mathfrak{sl}(\infty), \mathfrak o(\infty), \mathfrak{sp}(\infty)$. We start our study of ideals $I\subset\operatorname{U}(\mathfrak g(\infty))$ by describing all possible associated pro-varieties of such ideals. The result is surprisingly simple and quite different from the case of a finite-dimensional $\mathfrak g$: these pro-varieties depend just on one integer (rank) $r$ and form a chain
$$0\subset\mathfrak g(\infty)^{\le 1}\subset\mathfrak g(\infty)^{\le 2}\subset...\subset\mathfrak g(\infty)^{\le r}\subset...\subset\mathfrak g(\infty)^*.$$

The next step is to describe explicitly all primitive ideals with a given associated pro-variety. This is where the results under review are not yet complete, i.e. such a description is known only for a certain class of primitive ideals.

Recall that a $\mathfrak g(\infty)$-module $M$ is {\it integrable}, if any element $g\in\mathfrak g(\infty)$ acts locally finitely on $M$. An ideal is {\it integrable} if it is the annihilator of an integrable $\mathfrak g(\infty)$-module. The study of integrable ideals was initiated by A.~Zhilinskii in the 1990's. Zhilinskii introduced the concept of a coherent local system of finite-dimensional $\mathfrak g(n)$-modules: a set $\{M_n\}$ of finite-dimensional $\mathfrak g(n)$-modules such that the isomorphism classes of $\{M_{n'}\}$ and $\{{M_n}|_{\mathfrak g(n')}\}$ coincide for $n'<n$. The main breakthrough of Zhilinskii was the classification of all coherent local systems of finite-dimensional $\mathfrak g(n)$-modules~\cite{Zh1, Zh2, Zh3}. This leads to a description of integrable primitive ideals: the final result concerning a correspondence between integrable primitive ideals in $\operatorname{U}(\mathfrak g(\infty))$ and simple coherent local systems of $\mathfrak g(n)$-modules is stated in~\cite{PP1}.

Another natural approach to primitive ideals is to compute the annihilators of simple highest weight $\mathfrak g(\infty)$-modules and to compare the resulting set of primitive ideals  with primitive ideals constructed by any other means. In particular, in analogy with Duflo's theorem one may ask whether any primitive ideal in $\operatorname{U}(\mathfrak g(\infty))$ is the annihilator of a simple highest weight module.

It is well known that splitting Borel subalgebras $\mathfrak b$ of ${\mathfrak g}(\infty)$ are not conjugate under the group 
$\operatorname{Aut}(\mathfrak g(\infty))$, and form infinitely many isomorphism classes. This leads to an 
enormous ``variety'' of simple highest weight ${\mathfrak{g}}(\infty)$-modules $\operatorname{L}_{\mathfrak{b}}(\lambda)$. 
Our first result is that for $\mathfrak g(\infty)=\mathfrak{sl}(\infty), \mathfrak{o}(\infty)$ all ideals of the form 
$\operatorname{Ann}_{\operatorname{U}(\mathfrak g(\infty))}\operatorname{L}_\mathfrak b(\lambda)$ are integrable.

For $\mathfrak{g}(\infty)=\mathfrak{sp}(\infty)$, the situation is slightly different. Here we see the first example of a nonintegrable primitive ideal: this is the annihilator of a highest weight Shale-Weil (oscillator) representation of $\mathfrak{sp}(\infty)$. 

In all three cases we provide an explicit criterion on a pair $(\mathfrak b, \lambda)$ for the ideal $\operatorname{Ann}_{\operatorname{U}(\mathfrak g(\infty))}\operatorname{L}_\mathfrak b(\lambda)$ to be nonzero. This result is new
for $\mathfrak g(\infty)=\mathfrak o(\infty), \mathfrak{sp}(\infty)$ (for $\mathfrak g(\infty)=\mathfrak{sl}(\infty)$ the analogous result is presented in our recent paper~\cite{PP2}) and its proof constitutes the most technical part of the present paper.
In particular, we rely on an algorithm which computes
the partition corresponding to the nilpotent orbit whose closure is the associated variety of a given highest weight module over a classical simple finite-dimensional Lie algebra.
This algorithm is extracted from the existing literature~\cite{Jo, Lusztig, Barbasch-Vogan}, and is presented in Subsection~\ref{SSrsa}.

Here are some corollaries of our results for $\mathfrak g(\infty)=\mathfrak{sl}(\infty), \mathfrak{o}(\infty), \mathfrak{sp}(\infty)$:

$\bullet$ any prime integrable ideal is primitive;

$\bullet$ a pair $(\mathfrak b, \lambda)$ has a simple numerical invariant, the rank of $\operatorname{Ann}_{\operatorname{U}(\mathfrak g(\infty))}\operatorname{L}_\mathfrak b(\lambda)$ (an explicit formula is yet to be written);

$\bullet$ a primitive ideal is the annihilator of a unique (up to isomorphism) simple module if and only if $I$ is the annihilator of a simple object in the category $\mathbb T_{\mathfrak g(\infty)}$.

Some open problems are:

a) Are all ideals of $\operatorname{U}(\mathfrak{sl}(\infty))$ and $\operatorname{U}(\mathfrak o(\infty))$ integrable?

b) Is the lattice of two-sided ideals of $\operatorname{U}(\mathfrak g(\infty))$ n\"otherian?

c) Is any primitive ideal the annihilator of a simple highest weight module?

d) Is it true that $I=I^2$ for any (integrable) ideal?

Finally, we prove that the lattices of two-sided ideals in $\operatorname{U}(\mathfrak o(\infty))$ and $\operatorname{U}(\mathfrak{sp}(\infty))$ are isomorphic. The isomorphism is provided by the $\mathfrak{osp}$-duality functor constructed by V.~Serganova in~\cite{S}.

{\bf Acknowledgements.} Both authors have been supported in part for the last 6 years through the DFG Priority Program ``Representation Theory''. Alexey Petukhov thanks also Jacobs University Bremen for its continued hospitality. Finally, we thank the referee for reading our manuscript very carefully.

\section{Locally simple Lie algebras}

We fix an algebraically closed field $\mathbb F$ of characteristic zero. All
vector spaces (including Lie algebras) are assumed to be defined
over $\mathbb F$. If $V$ is a vector space, $V^*$ stands for the
dual space Hom$_\mathbb F(V, \mathbb F)$. All varieties we consider are algebraic varieties over $\mathbb F$ (with Zariski topology).
An ideal in a noncommutative ring always means a two-sided ideal.

By a {\it locally simple Lie algebra} we understand the inductive limit $\varinjlim \mathfrak g(n)$ of a chain~(\ref{Ech1}) of simple finite-dimensional Lie algebras.
The sign $\subset$ denotes not necessarily strict inclusion. By definition, a {\it natural
representation} (or a {\it natural module}) of a classical simple finite-dimensional Lie
algebra is a simple non-trivial finite-dimensional representation of minimal
dimension. When considering locally finite Lie algebras or their enveloping algebras we assume
that any given chain~(\ref{Ech1}) consists of inclusions, so we can freely
interchange $\lim\limits_{\to}\mathfrak g(n)$ with $\cup_n\mathfrak g(n)$, and $\lim\limits_{\to}\operatorname{U}(\mathfrak g(n))$ with $\cup_n\operatorname{U}(\mathfrak g(n))$, where $\operatorname{U}(\cdot)$ stands for enveloping algebra.


The most basic examples of locally simple Lie algebras are the three simple Lie
algebras $\mathfrak{sl}(\infty), \mathfrak o(\infty)$ and $\mathfrak{sp}(\infty)$.
These Lie algebras can be defined as respective unions of classical finite-dimensional Lie algebras of a fixed type
$\mathfrak{sl}, \mathfrak o,$ or $\mathfrak{sp}$ under the inclusions which arise from extending a natural representation by 1-dimensional increments for $\mathfrak{sl}$ and $\mathfrak o$, and by 2-dimensional increments for $\mathfrak{sp}$. An important result, see~\cite{Ba1}
or~\cite{BS}, states that, up to isomorphism, these three Lie
algebras are the only locally simple {\it finitary} Lie algebras, i.e.
locally simple Lie algebras which admit a countable-dimensional
faithful module with a basis such that the endomorphism arising from
each element of the Lie algebra is given by a matrix with finitely
many nonzero entries. In Appendix~A we give a precise definition of the Lie algebras $\mathfrak{sl}(\infty), \mathfrak o(\infty)$, $\mathfrak{sp}(\infty)$, and write down explicit bases for them.

A very interesting class of locally finite locally simple Lie algebras are the diagonal locally finite Lie algebras introduced by Y.~Bahturin and H.~Strade in~\cite{BhS}. We recall that an inclusion
$\mathfrak g(i)\subset \mathfrak g(j)$ of simple classical Lie algebras of the same type $\mathfrak{sl}, \mathfrak{o}, \mathfrak{sp}$, is {\it diagonal}
if the restriction $V(j)|_{\mathfrak g(i)}$ of a natural representation $V(j)$ of $\mathfrak g(j)$ to $\mathfrak g(i)$ is isomorphic
to a direct sum of copies of a natural $\mathfrak g(i)$-representation $V(i)$, of its dual $V(i)^*$, and
of the trivial 1{-}dimensional $\mathfrak g(i)$-representation. In this paper, by a {\it diagonal Lie
algebra} $\mathfrak g(\infty)$ we mean an infinite-dimensional Lie algebra obtained as the union $\cup_n\mathfrak g(n)$ of classical simple Lie
algebras $\mathfrak g(i)$ under diagonal inclusions $\mathfrak g(n)\subset\mathfrak g(n+1)$.  In~\cite{BZh1} A.~Baranov and A.~Zhilinskii have provided a rather complicated but explicit classification of isomorphism classes of diagonal locally simple Lie algebras. The three finitary locally simple Lie algebras are of course diagonal. An example of a diagonal nonfinitary Lie algebra is the Lie algebra $\mathfrak{sl}(2^\infty)$: by definition, $\mathfrak{sl}(2^\infty)=\varinjlim\mathfrak{sl}(2^n)$ for the chain of inclusions $\mathfrak{sl}(2^n)\subset\mathfrak{sl}(2^{n+1})$,$$A\mapsto\left(\begin{array}{cc}A&0\\0&A\end{array}\right).$$


\section{Associated pro-varieties of ideals}
Let $\mathfrak g(\infty)$ be a locally simple Lie algebra. We think of $\mathfrak g(\infty)$ as a direct limit of a fixed chain of Lie algebras~(\ref{Ech1}).  We consider ideals $I$ in the enveloping algebra $\operatorname{U}(\mathfrak g(\infty))$. We say that $I$ has {\it locally finite codimension} if the ideals $\operatorname{U}(\mathfrak g(n))\cap I$ have finite codimension in $\operatorname{U}(\mathfrak g(n))$ for all $n>0$.

In this section we outline our approach to the proof of the following theorem.
\begin{theorem}[\cite{PP1}]\label{T31} Let $\mathfrak g(\infty)$ be a locally simple Lie algebra. If $\operatorname{U}(\mathfrak g(\infty))$ admits a nonzero ideal of locally infinite codimension, then $\mathfrak g(\infty)\cong\mathfrak{sl}(\infty), \mathfrak o(\infty)$, $\mathfrak{sp}(\infty)$.\end{theorem}

We provide a sketch of proof of Theorem~\ref{T31} in Subsection~\ref{SSonpr}. Theorem~\ref{T31} is closely connected to the following result, previously conjectured by A.~Baranov.

\begin{theorem}[\cite{PP1}] \label{T32}If $\mathfrak g(\infty)$ is not (isomorphic to) a diagonal Lie algebra, then the augmentation ideal is the only nonzero proper ideal of $\operatorname{U}(\mathfrak g(\infty))$.\end{theorem}

Theorem~\ref{T32} is implied by Theorem~\ref{T31} by use of the following result proved by A.~Zhilinskii.

\begin{theorem}[\cite{Zh2}]If, for a locally simple Lie algebra $\mathfrak g(\infty)$, the algebra $\operatorname{U}(\mathfrak g(\infty))$ admits an ideal $I$ of locally finite codimension, then $\mathfrak g(\infty)$ is diagonal.\end{theorem}
Zhilinskii's proof is based on a notion of coherent local systems of modules for $\mathfrak g(\infty)$ which we review in Section~\ref{Sclsint}.

\subsection{Associated varieties and Poisson ideals}\label{SSavpi}
Let $\mathfrak g$ be a (finite- or infinite-dimensional) Lie algebra and $I\subset\operatorname{U}(\mathfrak g)$ be an ideal in the enveloping algebra
$\operatorname{U}(\mathfrak g)$ of $\mathfrak g$. The degree filtration \{$\operatorname{U}(\mathfrak g)^{\le
d}\}_{d\in\mathbb Z_{\ge0}}$ on $\operatorname{U}(\mathfrak g)$ defines the filtration $$\{I\cap \operatorname{U}(\mathfrak g)^{\le
d}\}_{d\in\mathbb Z_{\ge0}}$$ on $I$. The associated graded object
$\operatorname{gr} I:=\oplus_d((I\cap\operatorname{U}(\mathfrak g)^{\le d})/(I\cap\operatorname{U}(\mathfrak g)^{\le
d-1}))$ is an ideal of $\operatorname{gr}(\operatorname{U}(\mathfrak g))=\operatorname{\bf S}^\cdot(\mathfrak g)$. We denote the set of zeros of $\operatorname{gr} I$ in $\mathfrak g^*$ by $\operatorname{Var}(I)\subset\mathfrak g^*$. The variety is called the {\it associated variety} of $I$. We denote by $\operatorname{rad}(\operatorname{gr} I)$ the radical of $\operatorname{gr} I$ and consider $\operatorname{\bf S}^\cdot(\mathfrak g)/\operatorname{rad}(\operatorname{gr} I)$ as ``the algebra of polynomial functions on $\operatorname{Var}(I)$''.

The symmetric algebra $\operatorname{\bf S}^\cdot(\mathfrak g)$ carries a natural adjoint action of $\mathfrak g$, and any ideal which is stable under this action is called {\it Poisson} (if $J$ is such an ideal then $\operatorname{\bf S}^\cdot(\mathfrak g)/J$ also carries a natural Poisson structure). It is clear that $\operatorname{gr} I$ is Poisson. If $\mathfrak g$ is a finite-dimensional Lie algebra or a locally simple Lie algebra, it is clear that $\operatorname{rad}(\operatorname{gr} I)$ is Poisson. This Poisson structure on $\operatorname{\bf S}^\cdot(\mathfrak g)$ is a powerful tool in the study of ideals of $\operatorname{U}(\mathfrak g)$.

If $\mathfrak g=\mathfrak g(\infty)$ is a locally simple Lie algebra, then $\operatorname{Var}(I)$ is a pro-variety, i.e. a projective limit of algebraic varieties. Indeed, fix a sequence~(\ref{Ech1}) and let
$\overline{\operatorname{pr}_{\mathfrak g(n)}\operatorname{Var}(I)}\subset\mathfrak g_n^*$
be the closure of the image of $\operatorname{Var}(I$) under the natural projection
$\operatorname{pr}_{\mathfrak g(n)}: \mathfrak g(\infty)^*\to\mathfrak g(n)^*$; by definition,
$\overline{\operatorname{pr}_{\mathfrak g(n)}\operatorname{Var}(I)}\subset\mathfrak g(n)^*$
is the set of zeros of $(\operatorname{gr} I)\cap\operatorname{\bf S}^\cdot(\mathfrak g(n))$ in $\mathfrak g(n)^*$. The space $\mathfrak g(\infty)^*$ equals the projective limit
$\varprojlim \mathfrak g(n)^*$, and therefore $\operatorname{Var}(I)\subset\mathfrak
g(\infty)^*$ is the projective limit of the algebraic varieties
$\overline{\operatorname{pr}_{\mathfrak g(n)}\operatorname{Var}(I)}$.

\subsection{On the proof of Theorem~\ref{T31}}\label{SSonpr} If an ideal $I\subset\operatorname{U}(\mathfrak g(\infty))$ is of locally infinite codimension, then the ideal gr$I\subset\operatorname{\bf S}^\cdot(\mathfrak g(\infty))$ is also of locally infinite codimension. Therefore Theorem~\ref{T31} follows from Theorem~\ref{Tlinf} below.
\begin{theorem}\label{Tlinf}Let $\mathfrak g(\infty)$ be any locally simple Lie algebra. If {\bf S}$^\cdot(\mathfrak g(\infty))$ admits a nonzero Poisson ideal of locally infinite codimension, then $\mathfrak g(\infty)=\mathfrak{sl}(\infty), \mathfrak o(\infty)$, $\mathfrak{sp}(\infty)$.\end{theorem}
This theorem is one of the main results of our work~\cite{PP1}. The following proposition is a key step in the proof.

\begin{proposition}\label{P321} Let $\mathfrak g$ be a finite-dimensional classical simple Lie algebra, $V$ be a natural $\mathfrak g$-module, and $\mathfrak g'\subset\mathfrak g$ be a simple Lie subalgebra of $\mathfrak g$. If there exists an adjoint orbit $\EuScript O\subset \mathfrak g^*$ such that its image in $\mathfrak g'^*$ is not dense, then
$$\dim (\mathfrak g'\cdot V)<2(\dim\mathfrak g'-\mathrm{rk}\mathfrak g') (\mathrm{rk}\mathfrak g'+1)~\mathrm{or}~2\dim\mathfrak g'+2\ge\dim V,$$
where $\mathfrak g'\cdot V$ is the sum of non-trivial simple $\mathfrak g'$-submodules of $V$.
\end{proposition}
The proof of Proposition~\ref{P321} is somewhat lengthy and we refer the reader directly to~\cite{PP1}.  Here we sketch the proof of the fact that Proposition~\ref{P321} implies Theorem~\ref{Tlinf}.
\begin{proof}[Sketch of proof of Theorem~\ref{Tlinf}] Denote by $G(n)$ the adjoint group of Lie algebra $\mathfrak g(n)$ for all $n\ge 1$. Let $J\subset\operatorname{\bf S}^\cdot(\mathfrak g(\infty))$ be a nonzero Poisson ideal of locally infinite codimension. Set $J_n:=J\cap\operatorname{\bf S}^\cdot(\mathfrak g(n))$ for any $n\ge1$.  Without loss of generality we can assume that $J$ is radical, as the radical of a Poisson ideal of locally infinite codimension in $\operatorname{\bf S}^\cdot(\mathfrak g(\infty))$ is again Poisson and of locally infinite codimension.

Fix $n$ so that $\operatorname{\bf S}^\cdot(\mathfrak g(n))\cap J$ is nonzero and of infinite codimension in $\operatorname{\bf S}^\cdot(\mathfrak g(n))$. Then the image of any $G(m+n)$-orbit under the morphism $\operatorname{Var}(J(m+n))\to \mathfrak g(n)^*$ is not dense in $\mathfrak g(n)^*$ since it lies in the proper closed subvariety $\operatorname{Var}(J(n))\subset\mathfrak g(n)^*$. Therefore Proposition~\ref{P321} implies that $\dim (\mathfrak g(n)\cdot V(m+n))$ is bounded by some function which depends on $n$ only. Hence the number of nontrivial simple $\mathfrak g(n)$-constituents in $V(m+n)$ and their dimensions are simultaneously bounded as $m$ grows.
This shows that the Dynkin index of the injections $\mathfrak g(n')\to\mathfrak g(n'+1)$ equals 1 for large enough $n'$, which implies that $\mathfrak g(\infty) $ is isomorphic to $\mathfrak{sl}(\infty), \mathfrak o(\infty)$ or $\mathfrak{sp}(\infty)$, see~\cite[proof of Theorem 3.1]{PP1}.
\end{proof}

\subsection{Associated pro-varieties of ideals in $\operatorname{U}(\mathfrak{sl}(\infty)), \operatorname{U}(\mathfrak o(\infty)), \operatorname{U}(\mathfrak{sp}(\infty)$)}
Fix now a Lie algebra $\mathfrak g(\infty)=\mathfrak{sl}(\infty), \mathfrak{o}(\infty)$, $\mathfrak{sp}(\infty)$ together with a chain~(\ref{Ech1}) such that $\lim\limits_{\to}\mathfrak g(n)=\mathfrak g(\infty)$. Without loss of generality we assume that for $n\ge 3$ all $\mathfrak g(n)$ are simple and of the same type $A, B, C$, or $D$, and that $\operatorname{rk}\mathfrak g(n)=n$. By $V(n)$ we denote a natural representation of $\mathfrak g(n)$ (for $\mathfrak g(n)$ of type A there are two choices for $V(n)$ up to isomorphism). We further assume that, for $n\ge 3$, $V(n+1)$ considered as a $\mathfrak g(n)$-module is isomorphic to $V(n)$ plus a trivial module.

Set
\begin{center}$\mathfrak g(n)^{\le r}:=\{X\in\mathfrak g(n)\mid$ there exists
$\lambda\in\mathbb F$ such that $\operatorname{rk} (X-\lambda$Id$_{V(n)})\le r\}$,
(2)\end{center}where $X$ is considered as a linear
operator on $V(n)$. Note that $\mathfrak g(n)^{\le r}$ is an algebraic
subvariety of $\mathfrak g(n)$ for a fixed $r$ and large enough $n$, see~\cite{PP1}. Choosing compatible identifications $\mathfrak
g_n\cong\mathfrak g_n^*$, we can assume that $\mathfrak g(n)^{\le
r}\subset\mathfrak g(n)^*$. Furthermore, for $\mathfrak
g(\infty)\cong\mathfrak{sl}(\infty), \mathfrak{o}(\infty), \mathfrak{sp}(\infty)$
one can check directly that the projection $\mathfrak g(n+1)^*\to\mathfrak
g(n)^*$ maps $\mathfrak g(n+1)^{\le r}$ surjectively to $\mathfrak
g(n)^{\le r}$. In this way we obtain a well-defined projective limit of algebraic
varieties $\mathfrak g(\infty)^{\le r}:=\varprojlim \mathfrak g(n)^{\le r}$.

The radical ideals $J_n^{\le r}$ of ${\bf S}^\cdot(\mathfrak g(n))$, with
respective zero-sets $\mathfrak g(n)^{\le r}\subset\mathfrak g_n^*$, form a
chain whose union we denote by $J^{\le r}$. The ideal $J^{\le r}$ is
a radical Poisson ideal of $\operatorname{\bf S}^\cdot(\mathfrak g(\infty))$. Moreover, the following result strengthens Theorem~\ref{Tlinf} by describing all radical Poison ideals in $\operatorname{\bf S}^\cdot(\mathfrak g(\infty))$.

\begin{theorem}[{~\cite[Theorem 3.3]{PP1}}]\label{Tlinfrk}Let $\mathfrak g(\infty)=\mathfrak{sl}(\infty), \mathfrak{o}(\infty), \mathfrak{sp}(\infty)$ and $J\subset\operatorname{\bf S}^\cdot(\mathfrak g(\infty))$ be a nonzero radical Poisson ideal. Then $J=J^{\le r}$ for some $r\in\mathbb Z_{\ge0}$.\end{theorem}
\begin{corollary} Let $\mathfrak g(\infty)=\mathfrak{sl}(\infty), \mathfrak o(\infty), \mathfrak{sp}(\infty)$ and $I\subset\operatorname{U}(\mathfrak g(\infty))$ be an ideal. Then $\operatorname{Var}(I)=\mathfrak g(\infty)^{\le r}$ for some $r\in\mathbb Z_{\ge0}$.\end{corollary}
\begin{proof} By Theorem~\ref{Tlinfrk}, we have rad(gr$I)=J^{\le r }$ for some $r\in\mathbb Z_{\ge0}$. Hence $$\operatorname{Var}(I)=\mathfrak g(\infty)^{\le r}.$$\end{proof}
We say that an ideal $I\subset\operatorname{U}(\mathfrak g(\infty))$ has {\it rank} $r\in\mathbb Z_{\ge0}$ if $\operatorname{Var}(I)=\mathfrak g(\infty)^{\le r}$.



\section{Coherent local systems and integrable ideals}\label{Sclsint}
In this section we review the concept of c.l.s. (introduced by A.~Zhilinskii) and show how this concept is related to two-sided ideals of $\operatorname{U}(\mathfrak g(\infty))$.

We consider a fixed locally simple Lie algebra $\mathfrak g(\infty)=\varinjlim \mathfrak g(n)$, and denote by $\operatorname{Irr} \mathfrak g(n)$ the set of isomorphism classes of simple finite-dimensional $\mathfrak g(n)$-modules.
\begin{definition}{\it A coherent local system of $\mathfrak g(n)$-modules} (further shortened as {\it c.l.s.}) for $\mathfrak g(\infty)=\varinjlim\mathfrak g(n)$ is a collection of sets \begin{center}$\{Q_n\}_{n\in\mathbb Z_{\ge1}}\subset\prod_{n\in\mathbb Z_{\ge1}}\operatorname{Irr}\mathfrak g(n)$\end{center} such that for any $n<m$ the following conditions hold:

$\bullet$ for any simple finite-dimensional module $M$ whose isomorphism class belongs to $Q_m$, the isomorphism classes of all simple constituents of $M|_{\mathfrak g(n)}$ belong to $Q_n$,

$\bullet$ for any simple finite-dimensional $\mathfrak g(n)$-module $N$ whose isomorphism class belongs to $Q_n$, there exists a simple finite-dimensional $\mathfrak g(m)$-module $M$ whose isomorphism class belongs to $Q_m$ and such that $N$ is isomorphic to a simple constituent of $M|_{\mathfrak g(n)}$.
\end{definition}

The c.l.s. for locally simple Lie algebras are classified by A.~Zhilinskii~\cite{Zh3}. A remarkable corollary of this classification is that, if $\mathfrak g(\infty)$ has a non-trivial c.l.s., then $\mathfrak g(\infty)$ is diagonal. This fact had led Baranov to his conjecture, see Theorem~\ref{T32} above.

Note that the c.l.s. for a given Lie algebra form a lattice with respect to the inclusion order (join equals union and meet equals intersection).

\subsection{Integrable ideals}
C.l.s. for $\mathfrak g(\infty)$ are related in a natural way to a special class of ideals of $\operatorname{U}(\mathfrak g(\infty))$ which we describe next. 
\begin{definition}(1) A $\mathfrak g(\infty)$-module
$M$ is {\it integrable} if, for any finitely generated subalgebra
$U'\subset\operatorname{U}(\mathfrak g(\infty))$ and any $m\in M$, we have
$\dim(U'\cdot m)<\infty$.

(2) An ideal $I\subset\operatorname{U}(\mathfrak g(\infty))$ is {\it integrable} if it is the annihilator of an integrable $\operatorname{U}(\mathfrak g(\infty))$-module.\end{definition}
This definition makes sense also for a finite-dimensional semisimple Lie algebra $\mathfrak g$. In that case integrable ideals are the annihilators of  arbitrary sums of finite-dimensional $\mathfrak g$-modules, and form a very special class of ideals. In the case of $\mathfrak g(\infty)$ integrable ideals play a much more prominent role.

To a c.l.s. $Q$ for $\mathfrak g(\infty)$ we attach the ideal
$$\operatorname{I}(Q):=\cup_n\cap_{V\in Q_n}(\operatorname{Ann}_{\operatorname{U}(\mathfrak g(n))}V).$$

\begin{lemma} An ideal $I\subset\operatorname{U}(\mathfrak g(\infty))$ is integrable if and only if $I=\operatorname{I}(Q)$ for some c.l.s. $Q$.\end{lemma}
\begin{proof} If an ideal $I$ is integrable, it is the annihilator of some integrable $\mathfrak g(\infty)$-module $M$. It is clear that $M$ determines a c.l.s. $Q_M$,
$$(\operatorname{Q}_M)_n:=\{\mbox{isomorphism~classes~of~simple~direct~summands~of~}M|_{\mathfrak{g}(n)}\},$$
and that $I=\mathrm{I}(Q_M)$.

Conversely, let $I=\operatorname{I}(Q)$ for some c.l.s. $Q$ for $\mathfrak g(\infty)$. For any $n\ge1$, let $V_n$ be the direct sum of representatives of the isomorphism classes in $Q_n$. The definition of c.l.s. guarantees that for any $n\ge1$ there exists an embedding $V_n\to V_{n+1}$ of $\mathfrak g(n)$-modules. Clearly, the direct limit of such embeddings is an integrable $\mathfrak g(\infty)$-module, and $I$ is the annihilator of this integrable module.
\end{proof}
A c.l.s. is called {\it irreducible} if it is not a union of proper sub-c.l.s.. Any c.l.s. is a finite union of irreducible c.l.s.~\cite{Zh1, Zh3}. Moreover, the following holds.

\begin{proposition}\label{P44} a) If $Q$ is an irreducible c.l.s. then $\operatorname{I}(Q)$ is a primitive ideal.

b) \label{Lprpr} An integrable ideal $I$ of $\operatorname{U}(\mathfrak g(\infty))$ is prime if and only if it is primitive.
\end{proposition}
\begin{proof} a) This result should be attributed to A.~Zhilinskii as it follows from~\cite[Lemma 1.1.2]{Zh1}.

b) This is proved in~\cite[Proposition 7.8]{PP1}.
\end{proof}
Next, to any ideal $I$ of $\operatorname{U}(\mathfrak g(\infty))$ we attach the c.l.s. $\operatorname{Q}(I)$ which is the largest c.l.s. such that $I\subset \operatorname{I}(\operatorname{Q}(I))$. The maps \begin{equation}Q\mapsto \operatorname{I}(Q)\hspace{10pt}\mbox{and}\hspace{10pt}I\mapsto \operatorname{Q}(I)\label{Eiqi}\end{equation} are not injective in general but they induce antiisomorphisms between interesting sublattices of the lattice of c.l.s. and of the lattice of integrable ideals.
\begin{proposition}\label{P44C}The maps~(\ref{Eiqi}) induce antiisomorphisms between the following lattices:

a) the lattice of c.l.s. of {\it finite type} (i.e., c.l.s. $Q$ such that all sets $Q_n$ are finite) and the lattice of ideals of $\operatorname{U}(\mathfrak g(\infty))$ of locally finite codimension, for any locally simple Lie algebra $\mathfrak g(\infty)$,

b) the lattice of c.l.s. and the lattice of ideals in $\operatorname{U}(\mathfrak g(\infty))$, for $\mathfrak g(\infty)$ diagonal and nonisomorphic to $\mathfrak{sl}(\infty), \mathfrak o(\infty), \mathfrak{sp}(\infty)$,

c) the lattice of c.l.s. and the lattice of integrable ideals in $\operatorname{U}(\mathfrak g(\infty))$, for $\mathfrak g(\infty)=\mathfrak o(\infty), \mathfrak{sp}(\infty)$.

\end{proposition}
\begin{proof} Part a) is an easy corollary of the following well-known fact: for any semisimple finite-dimensional Lie algebra $\mathfrak g$, there is a natural bijection between the lattice of ideals of finite codimension in $\operatorname{U}(\mathfrak g)$ and the lattice of finite sets of isomorphism classes of finite-dimensional $\mathfrak g$-modules.

Part b) is implied by part a) and the followings two facts:

$\bullet$ according to~\cite{Zh3}, if $\mathfrak g(\infty)$ is diagonal and $\mathfrak g(\infty)\not\cong\mathfrak{sl}(\infty), \mathfrak o(\infty), \mathfrak{sp}(\infty),$ then any c.l.s. of $\mathfrak g(\infty)$ is of finite type,

$\bullet$ under the same assumptions, any ideal of $\operatorname{U}(\mathfrak g(\infty))$ is of locally finite codimension, see Theorem~\ref{T31}.

Part c) is a restatement of \cite[Theorem~7.9b)]{PP1}, see the proof there.
\end{proof}
\begin{remark}\label{R45} For integrable ideals, the map $I\mapsto \operatorname{Q}(I)$ is always injective. For $\mathfrak{sl}(\infty)$, the map $I\mapsto \operatorname{Q}(I)$ is not bijective. Theorem 7.9 a) in~\cite{PP1} describes a set of irreducible c.l.s., called {\it left c.l.s.}, such that the map $Q\mapsto \operatorname{I}(Q)$ induces a bijection between left c.l.s. for $\mathfrak{sl}(\infty)$ and integrable ideals of $\operatorname{U}(\mathfrak{sl}(\infty))$. However, it is easy to see that
this bijection cannot be extended to an antiisomorphism of lattices, see~\cite{PP1}. We skip the definition of left c.l.s. in this paper, and refer the reader to~\cite{PP1}.\end{remark}
\begin{remark}It seems plausible that all ideals of $\operatorname{U}(\mathfrak{sl}(\infty))$ and $\operatorname{U}(\mathfrak o(\infty))$ are integrable. If this is so, then $\operatorname{U}(\mathfrak{sl}(\infty))$ and $\operatorname{U}(\mathfrak{o}(\infty))$ will have countable many ideals. In addition, $\operatorname{U}(\mathfrak{sp}(\infty))$ will also have countable many ideals by Theorem~\ref{Tilosp} below.\end{remark}
\subsection{Classification of prime integrable ideals for finitary Lie algebras} In the rest of this paper $\mathfrak g(\infty)=\mathfrak{sl}(\infty), \mathfrak o(\infty)$ or $\mathfrak{sp}(\infty)$. Any c.l.s. is a union of finitely many irreducible c.l.s., and thus any integrable ideal is an intersection of finitely many primitive or, equivalently, prime integrable ideals. Therefore, a description of prime integrable ideals is a basis for a description of all integrable ideals. In this subsection we assume that $\mathfrak g(\infty)=\mathfrak{sl}(\infty), \mathfrak o(\infty)$ or $\mathfrak{sp}(\infty)$, and describe the prime integrable ideals of $\operatorname{U}(\mathfrak g(\infty))$ as annihilators of certain integrable $\mathfrak g(\infty)$-modules.

We define a {\it natural} $\mathfrak g(\infty)$-module $V(\infty)$ as a direct limit $\varinjlim_{n} V(n)$ of natural $\mathfrak g(n)$-modules. Such a limit is unique (up to isomorphism) for $\mathfrak g(\infty)=\mathfrak{o}(\infty), \mathfrak{sp}(\infty)$, while for $\mathfrak g(\infty)=\mathfrak{sl}(\infty)$ (up to isomorphism) there are two natural modules: $V(\infty)$ and $V(\infty)_*$. These are twists of each other by the Cartan involution of $\mathfrak{sl}(\infty)$.  
We set also\begin{equation}\begin{array}{cccc}{\bf \Lambda}^p:={\bf \Lambda}^p(V(\infty)),&\operatorname{\bf S}^p:=\operatorname{\bf S}^p(V(\infty)),&{\bf \Lambda}^\cdot:={\bf \Lambda}^\cdot(V(\infty)),&\operatorname{\bf S}^\cdot:=\operatorname{\bf S}^\cdot(V(\infty)),\\
{\bf \Lambda}^p_*:={\bf \Lambda}^p(V_*(\infty)),&\operatorname{\bf S}^p_*:=\operatorname{\bf S}^p(V_*(\infty)),&{\bf \Lambda}^\cdot_*:={\bf \Lambda}^\cdot(V_*(\infty)),&\operatorname{\bf S}^\cdot_*:=\operatorname{\bf S}^\cdot(V_*(\infty)),\\\end{array}\label{Ebcls}\end{equation}
where $p\in\mathbb Z_{\ge0}$, and $\Lambda^\cdot$ stands for exterior algebra. In addition, for $\mathfrak g(\infty)=\mathfrak{o}(\infty)$ we let $\operatorname{Spin}$ to be a fixed simple $\mathfrak{o}(\infty)$-module which is an inductive limit of simple spinor modules of $\mathfrak o(2n+1)$ for $n\to\infty$. Such a module is not unique up to isomorphism.

Zhilinskii has introduced the notion of {\it basic irreducible c.l.s.}: in our language these are the c.l.s. of the modules~(\ref{Ebcls}) and the c.l.s. of the $\mathfrak o(\infty)$-module $\operatorname{Spin}$. Zhilinskii proves that any irreducible c.l.s. can be represented canonically in terms of a certain product of basic c.l.s., which he calls Cartan product~\cite{Zh1}. The notion of Cartan product and Zhilinskii's decomposition of an arbitrary c.l.s. are recalled in our paper~\cite[Section 7.1-7.2]{PP1}.

Furthermore, for any Young diagram $Y$ (possibly empty) whose column lengths form a sequence  $l_1\ge l_2\ge...\ge l_s>0$, we define the $\frak g(\infty)$-module $V^Y$ as the direct limit $\varinjlim_{n\ge s} V^Y(n)$ where $V^Y(n)$ denotes the simple finite-dimensional $\frak g(n)$-module with highest weight $l_1\varepsilon_1+...+l_s\varepsilon_s$~(the vectors $\varepsilon_1,..., \varepsilon_s$ are introduced in Appendix~A; for $Y=\emptyset$ the highest weight of $V^Y(n)$ equals 0). The $\frak g(n)$-module $V^Y(n)$ is isomorphic to a simple direct summand of the tensor product $$\operatorname{\bf S}^{l_1}(V(n))\otimes\operatorname{\bf S}^{l_2}(V(n))\otimes....\otimes\operatorname{\bf S}^{l_s}(V(n)),$$ and the above direct limit is clearly well defined up to isomorphism. Similarly, for $\frak g(\infty)=\frak{sl}(\infty)$, we define $V^Y_*$ as the direct limit $\varinjlim_{n\ge s}(V^Y(n))^*$.

The following classification of prime integrable ideals is closely related to Zhilinskii's classification of irreducible c.l.s. (the classification of all integrable ideals is a little bit more involved, see~\cite[Theorem~7.9]{PP1}).
\begin{proposition}\label{Pclsm}a) Any nonzero prime integrable ideal $I\subsetneq \operatorname{U}(\mathfrak g(\infty))$ is the annihilator of a unique $\frak g(\infty)$-module of the form
$$\begin{array}{cc}V^{Y_l}\otimes({\bf \Lambda}^\cdot)^{\otimes v}\otimes(\operatorname{\bf S}^\cdot)^{\otimes w}\otimes V^{Y_r}_*&\mbox{for $\mathfrak g(\infty)=\mathfrak{sl}(\infty)$},\\
V^{Y_l}\otimes ({\bf \Lambda}^\cdot)^{\otimes v}\otimes(\operatorname{\bf S}^\cdot)^{\otimes w}&\mbox{for $\mathfrak g(\infty)=\mathfrak{sp}(\infty)$},\\
\begin{array}{c}V^{Y_l}\otimes ({\bf \Lambda}^\cdot)^{\otimes v}\otimes(\operatorname{\bf S}^\cdot)^{\otimes w}\\\mbox{or}\\V^{Y_l}\otimes({\bf \Lambda}^\cdot)^{\otimes v}\otimes(\operatorname{\bf S}^\cdot)^{\otimes w}\otimes\operatorname{Spin}\end{array}&\mbox{for $\mathfrak g(\infty)=\mathfrak o(\infty)$},\\
\end{array}$$
where $v, w\in\mathbb Z_{\ge0},$ and $Y_l, Y_r$ are arbitrary Young diagrams.\\
b) If $I$ is the annihilator of the respective module in a), then the rank of $I$ equals $w$ for $\mathfrak g(\infty)=\mathfrak{sl}(\infty)$, and $2w$ for $\mathfrak g(\infty)=\mathfrak o(\infty), \mathfrak{sp}(\infty)$.
\end{proposition}
\begin{proof} As the c.l.s. of the modules in the statement of Proposition~\ref{Pclsm} can be computed explicitly, see~\cite[Theorem~2.3]{PSt}, it is relatively straightforward to compare the statement of Proposition~\ref{Pclsm} with Zhilinskii's description of irreducible c.l.s.~\cite{Zh1}. This, together with Proposition~\ref{P44C}c) and Remark~\ref{R45}, implies a).

Part b) follows from~\cite[Section 7, formula (9)]{PP1}.\end{proof}
\subsection{$(\operatorname{\bf S}-{\bf \Lambda})$--involution and $\mathfrak{osp}$-duality} In the paper~\cite{DPS} (and independently in~\cite{SSn}) a category $\mathbb T_{\mathfrak g(\infty)}$ of tensor modules has been introduced for $\mathfrak g(\infty)=\mathfrak{sl}(\infty), \mathfrak o(\infty), \mathfrak{sp}(\infty)$. This category is analogous to the category of finite-dimensional modules over a finite-dimensional Lie algebra, and is proven to be Koszul but not semisimple.  Moreover, there is an equivalence of the tensor categories $\mathbb T_{\mathfrak{o}(\infty)}$ and $\mathbb T_{\mathfrak{sp}(\infty)}$~\cite{DPS, SSn, S}, and we refer to this equivalence as $\mathfrak{osp}$-duality. This duality identifies the natural modules $V(\infty)$ for both Lie algebras but sends the symmetric powers $\operatorname{\bf S}^k(V(\infty))$ for one Lie algebra to the exterior powers ${\bf \Lambda}^k(V(\infty))$ for the other (in particular, it identifies the adjoint representations for $\mathfrak{o}(\infty)$ and $\mathfrak{sp}(\infty)$).

There exists a similarly defined involutive tensor functor on the category of tensor modules $\mathbb T_{\mathfrak{sl}(\infty)}$, and it also interchanges $\operatorname{\bf S}^k(V(\infty))$ and ${\bf \Lambda}^k(V(\infty))$~\cite{S}. 

In Appendix B we prove the following version of $\mathfrak{osp}$-duality.
\begin{theorem}\label{Tilosp}\label{T49}There is an isomorphism between the lattices of ideals in $\operatorname{U}(\mathfrak{o}(\infty))$ and in $\operatorname{U}(\mathfrak{sp}(\infty))$.\end{theorem}

If $Y$ is a Young diagram, by $Y'$ we denote the conjugate Young diagram, i.e. the Young diagram whose column lengths equal the row lengths of $Y$. The isomorphism from Theorem~\ref{Tilosp} identifies the annihilators of the modules
$$V^{Y}\otimes({\bf \Lambda}^\cdot)^{\otimes v}\otimes(\operatorname{\bf S}^\cdot)^{\otimes w}\mbox{~and~}V^{Y'}\otimes({\bf \Lambda}^\cdot)^{\otimes w}\otimes (\operatorname{\bf S}^\cdot)^{\otimes v},$$
where one module is an $\mathfrak o(\infty)$-module and the other is an $\mathfrak{sp}(\infty)$-module.
Under the isomorphism of Theorem~\ref{Tilosp}, the annihilator of $\operatorname{Spin}$ (this annihilator is the kernel of the canonical homomorphism from $\operatorname{U}(\mathfrak o(\infty))$ to the Clifford algebra of $V(\infty)$) goes to the annihilator of a Shale-Weil module (this annihilator is the kernel of the canonical homomorphism from $\operatorname{U}(\mathfrak{sp}(\infty))$ to the Weyl algebra of $V(\infty)$).

For $\mathfrak{sl}(\infty)$, the corresponding involutive tensor functor identifies the annihilators of the $\mathfrak{sl}(\infty)$-modules
\begin{center}$V^{Y_l}\otimes[({\bf \Lambda}^\cdot)^{\otimes v}\otimes(\operatorname{\bf S}^\cdot)^{\otimes w}]\otimes V^{Y_r}_*$\end{center}and\begin{center}
$V^{Y_l'}\otimes[({\bf \Lambda}^\cdot)^{\otimes w}\otimes(\operatorname{\bf S}^\cdot)^{\otimes v}]\otimes V^{Y_r'}_*$.\end{center}


\section{Annihilators of highest weight $\mathfrak g(\infty)$-modules}\label{Sann}
We now present some results on the annihilators of simple highest weight modules of $\mathfrak g(\infty)=\mathfrak{sl}(\infty), \mathfrak o(\infty), \mathfrak{sp}(\infty)$. The notion of highest weight module is based on the notion of a splitting Borel subalgebra of $\mathfrak g(\infty)$, and in Appendix~A we have collected the necessary preliminaries. Very roughly, our main result in this direction is that most simple highest weight modules have trivial annihilator, and that the few ones that have a nontrivial annihilator are either integrable or very similar to integrable.

\subsection{Splitting Borel and Cartan subalgebras}\label{SSspl} First we fix the chain~(\ref{Ech1}) to be of the form:
$$\begin{array}{c|ccccccccc}A&\mathfrak{sl}(2)&\to&\mathfrak{sl}(3)&\to&...&\to&\mathfrak{sl}(n+1)&\to&...\\
B&\mathfrak o(3)&\to&\mathfrak o(5)&\to&...&\to&\mathfrak o(2n+3)&\to&...\\
C&\mathfrak{sp}(2)&\to&\mathfrak{sp}(4)&\to&...&\to&\mathfrak{sp}(2n)&\to&...\\
D&\mathfrak o(6)&\to&\mathfrak o(8)&\to&...&\to&\mathfrak o(2n+4)&\to&...\\
\end{array}.$$
Clearly, the chain $A$ corresponds to Lie algebra $\mathfrak{sl}(\infty)$, the chains $B$ and $D$ correspond to $\mathfrak o(\infty)$, and the chain $C$ corresponds to $\mathfrak{sp}(\infty)$.

One can pick Cartan subalgebras $\mathfrak h(n)\subset\mathfrak g(n)$ in such a way that the image of $\mathfrak h(n)$ under the map $\mathfrak g(n)\to\mathfrak g(n+1)$ lies in $\mathfrak h(n+1)$. Then we have a well-defined inductive limit $\mathfrak h:=\varinjlim \mathfrak h(n)$. The Lie algebra $\mathfrak h$ is a maximal commutative subalgebra of $\mathfrak g(\infty)$, and is a {\it splitting Cartan subalgebra} of $\mathfrak g(\infty)$~\cite{DPSn}. It is known that in $\mathfrak{sl}(\infty)$ and $\mathfrak{sp}(\infty)$ a splitting Cartan subalgebra is unique up to conjugation via the group $\operatorname{Aut}(\mathfrak g(\infty))$~\cite{DPSn}. In $\mathfrak o(\infty)$ there are two conjugacy classes of splitting Cartan subalgebras, see~\cite{DPSn} or Appendix~A. In the rest of this paper we fix splitting Cartan subalgebras $\mathfrak h^A\subset\mathfrak{sl}(\infty), \mathfrak h^C\subset\mathfrak{sp}(\infty), \mathfrak h^B, \mathfrak h^D\subset\mathfrak{o}(\infty)$. The latter two subalgebras belong to different conjugacy classes and arise respectively from the above sequences $B$ and $D$.

Any maximal locally solvable subalgebra $\mathfrak b\subset\mathfrak g(\infty)$ which contains a splitting Cartan subalgebra is called {\it a splitting Borel subalgebra}. We can assume that $\mathfrak b$ contains $\mathfrak h^A, \mathfrak h^B, \mathfrak h^C$ or $\mathfrak h^D$. Any linear order $\prec$ on $\mathbb Z_{>0}$ defines a splitting Borel subalgebra $\mathfrak b(\prec)~(\mathfrak b\supset\mathfrak h^{A/B/C/D}$): this is explained in Appendix~A. Moreover, any conjugacy class of pairs (splitting Borel subalgebra, splitting Cartan subalgebra) contains a pair $(\mathfrak b(\prec), \mathfrak h^{A/B/C/D})$ defined by a suitable order $\prec$. Thus, from now on, we fix a linear order $\prec$ on $\mathbb Z_{>0}$ and pick a Borel subalgebra \begin{center}$\mathfrak b:=\mathfrak b(\prec)$, $\mathfrak b\supset\mathfrak h^{A/B/C/D}$,\end{center} corresponding to this order.

Let $\mathbb Z_{>0}=S_1\sqcup ...\sqcup S_t$ be a finite partition of $\mathbb Z_{>0}$. We say that this partition is {\it compatible} with the order $\prec$ if, for any $i\ne j\le t$, \begin{center}$i<j\Rightarrow i_0\prec j_0$\end{center}for all $i_0 \in S_i, j_0\in S_j$.

\begin{definition}We call a splitting Borel subalgebra $\mathfrak b\supset \mathfrak h^{A/B/C/D}$ of $\mathfrak g$ {\it ideal} if it satisfies the following conditions\\
$A$-case: there exists a partition $\mathbb Z_{>0}=S_1\sqcup S_2\sqcup S_3$, compatible with the order $\prec$ defined by $\mathfrak b$, such that

$\bullet$ $S_1$ is countable, and $\prec$ restricted to $S_1$ is isomorphic to the standard order on $\mathbb Z_{>0}$.

$\bullet$ $S_3$ is countable, and $\prec$ restricted to $S_3$ is isomorphic to the standard order on $\mathbb Z_{<0}$.\\
$B/C/D$-cases: there exists a partition $\mathbb Z_{\ge0}=S_1\sqcup S_2$, compatible with the order $\prec$ defined by $\mathfrak b$, such that

$\bullet$ $S_1$ is countable, and $\prec$ restricted to $S_1$ is isomorphic to the standard order on $\mathbb Z_{>0}$.
\end{definition}

\subsection{Almost integral and almost half-integral  weights}
Let $\mathbb F^{\mathbb Z_{>0}}$ denote the set of functions from $\mathbb Z_{>0}$ to $\mathbb F$. For $f\in\mathbb F^{\mathbb Z_{>0}}$, by $|f|$ we denote the cardinality of the image of $f$. There is a morphism from $\mathbb F^{\mathbb Z_{>0}}$ to $\mathfrak h^*$ :\begin{equation}f\mapsto \lambda_f,\hspace{10pt}
\lambda_f(e_{i, -i})=f(i);\end{equation}
here $e_{i, -i}$ is some explicitly given basis element of $\mathfrak h^{B/C/D}$, see Appendix~A.
This map is surjective in all cases and is an isomorphism in the $B/C/D$-cases. 

\begin{definition}$A$-case: A function $f\in\mathbb F^{\mathbb Z_{>0}}$ is {\it integral} if $f(i)-f(j)\in\mathbb Z$ for all $i, j\in {\mathbb Z_{>0}}$, and is {\it almost integral} if $f(i)-f(j)\in\mathbb Z$ for all $i, j\in{\mathbb Z_{>0}}\backslash F$ for some finite set $F\subset{\mathbb Z_{>0}}$.

$B/C/D$-cases: A function $f\in\mathbb F^{\mathbb Z_{>0}}$ is {\it integral} (respectively, {\it half-integral}) if $f(i)\in\mathbb Z$ (respectively, $f(i)\in\mathbb Z+\frac12$) for all $i\in S$, and is {\it almost integral} (respectively, {\it almost half-integral}) if $f(i)\in\mathbb Z$ (respectively, $f(i)\in\mathbb Z+\frac12$) for all $i\in{\mathbb Z_{>0}}\backslash F$ for some finite subset $F\subset\mathbb Z_{>0}$.\end{definition}

Finally, we say that $f\in \mathbb F^{\mathbb Z_{>0}}$ is {\it locally constant with respect to} $\prec$ if there exists
a compatible partition $\mathbb Z_{>0}=S_1\sqcup...\sqcup S_t$ such
that $f$ is constant on $S_i$ for any $i\le t$.

\subsection{Main results}\label{SStmr} Let $\mathfrak h\subset\mathfrak g(\infty)=\mathfrak{sl}(\infty), \mathfrak o(\infty), \mathfrak{sp}(\infty)$ be a splitting Cartan subalgebra as in Subsection~\ref{SSspl}, and $\mathfrak b$ be a splitting Borel subalgebra. The map $\mathfrak h\to \mathfrak b/[\mathfrak b, \mathfrak b]$ is an isomorphism, hence any weight $\lambda\in\mathfrak h^*$ defines a character $\lambda:\mathfrak b\to\mathbb F$ or, equivalently, a 1-dimensional $\mathfrak b$-module $\mathbb F_\lambda$. We denote by $\operatorname{L}_\mathfrak b(\lambda)$ the unique simple quotient of the Verma module $\operatorname{M}_\mathfrak b(\lambda):=\operatorname{U}(\mathfrak g(\infty))\otimes_{\operatorname{U}(\mathfrak b)}\mathbb F_\lambda$. Put $\operatorname{L}_\mathfrak b(f):=\operatorname{L}_\mathfrak b(\lambda_f), \operatorname{M}_\mathfrak b(f):=\operatorname{M}_\mathfrak b(\lambda_f)$.

In the A-case, the following results have appeared in~\cite{PP2}.
\begin{theorem}\label{1Tbcd} Let $\prec$ be some order on $\mathbb Z_{>0}$, $\mathfrak b\supset\mathfrak h$ be the respective splitting Borel subalgebra of $\mathfrak g(\infty)$, and $f\in\mathbb F^{\mathbb Z_{>0}}$. Then $$\operatorname{Ann}_{\operatorname{U}(\mathfrak g(\infty))} \operatorname{L}_{\mathfrak b}(f)\ne 0$$ if and only if

(1) $f$ is almost integral in the $A$-case and $f$ is almost integral or almost half-integral in the $B/C/D$-cases,

(2) $f$ is locally constant with respect to $\prec$.\end{theorem}
\begin{theorem}[$A/B/D$-cases]\label{2T} The following conditions on a nonzero ideal $I$ of $\operatorname{U}(\mathfrak g(\infty))$ are equivalent:

--- $I=\operatorname{Ann}_{\operatorname{U}(\mathfrak g(\infty))} \operatorname{L}_{\mathfrak b}(f)$ for some splitting Borel subalgebra $\mathfrak b\supset\mathfrak h$ and some function $f\in\mathbb F^{\mathbb Z_{>0}}$;

--- $I$ is a prime integrable ideal of $\operatorname{U}(\mathfrak g(\infty))$;

--- $I=\operatorname{Ann}_{\operatorname{U}(\mathfrak g(\infty))} \operatorname{L}_{\mathfrak b^0}(f^0)$ for some $f^0\in\mathbb F^{\mathbb Z_{>0}}$, where $\mathfrak b^0$ is any fixed ideal Borel subalgebra.\\
\end{theorem}
\begin{proposition}\label{Pp} If $\mathfrak b$ is a nonideal Borel subalgebra then there exists a prime integrable ideal $I$ which does not arise as the annihilator of a simple $\mathfrak b$-highest weight $\mathfrak g(\infty)$-module.\end{proposition}

The proofs of Theorems~\ref{1Tbcd},~\ref{2T} and Proposition~\ref{Pp} for the B/C/D-cases are given in Section~\ref{S6} below.

\subsection{The annihilators of simple integrable highest weight modules} We should point out that Theorems~\ref{1Tbcd}-\ref{2T} come short of an explicit computation of the annihilator $\operatorname{I}_\mathfrak b(f)$ of a given simple highest weight module $\operatorname{L}_\mathfrak b(f)$. 
In this subsection we present an explicit formula for $\operatorname{I}_\mathfrak b(f)$ under the assumption that the
 $\mathfrak g(\infty)$-module $\operatorname{L}_\mathfrak b(f)$  is integrable.

Set $\operatorname{I}_\mathfrak b(f):=\operatorname{Ann}_{\operatorname{U}(\mathfrak g(\infty))}\operatorname{L}_\mathfrak b(f)$. The following lemma is straightforward.
\begin{lemma}\label{L56} Let $f\in\mathbb F^{\mathbb Z_{>0}}$ be a function and $\mathfrak b$ be a splitting Borel subalgebra of $\mathfrak g(\infty)$ such that $\mathfrak b\supset\mathfrak h^{A/B/C/D}$. The following conditions are equivalent:

$\bullet$ $\operatorname{L}_{\mathfrak b}(f)$ is an integrable $\mathfrak g(\infty)$-module,

$\bullet$ $f$ is $\mathfrak b$-dominant  (see Appendix~A for the definition).\end{lemma}

We pick a linear order $\prec$ on $\mathbb Z_{>0}$, and thus a Borel subalgebra $\mathfrak b\subset\mathfrak g(\infty)$ such that $\mathfrak b\supset\mathfrak h^{A/B/C/D}$. We also pick a $\mathfrak b$-dominant function $f\in\mathbb F^{\mathbb Z_{>0}}$. Theorem~\ref{1Tbcd} implies that if $|f|=\infty$ then $\operatorname{I}_{\mathfrak b}(f)=0$. Thus, from now on, we assume that $|f|<\infty$.

The equivalent conditions from Lemma~\ref{L56} imply that in the $C$-case $f$ is integral, and that in the $A$-case we can assume without loss of generality $f$ has integer values. In the $B/D$-cases we can assume that the values of $f$ are positive. Here we have to consider two different subcases: $f$ is integral, $f$ is half-integral.
In all cases the maximal and minimal value of $f$ are well defined: we denote them by $a$ and $b$ respectively. 
For any  $c\in\mathbb Z\cup (\mathbb Z+\frac12)$ we let

$|\le c|$ be the cardinality of the subset of $f^{-1}([b, c])\subset \mathbb Z_{>0}$,

$p$ be the smallest integer or half-integer such that $|\le p|=+\infty$,

$|\ge c|$ be the cardinality of the subset of $f^{-1}([c, a])\subset \mathbb Z_{>0}$,

$q$ be the largest integer of half-integer such that $|\ge q|=+\infty$.\\
By $Y_r(f)$ we denote the Young diagram whose sequence of row lenghts equals the sequence
$$|\le(p-1)|\ge |\le(p-2)|\ge ...\ge |\le b|>0$$
if $|\le b|\in\mathbb Z_{>0}$; in case $|\le b|=\infty$, we set $Y_r(f):=\emptyset$. Finally, let $Y_l(f)$ be the Young diagram whose sequence of row lenghts equals the sequence
$$|\ge(q+1)|\ge|\ge q|\ge ...\ge |\ge a|>0$$
for $|\ge a|\in\mathbb Z_{>0}$; in case $|\ge a|=\infty$, we set $Y_l(f):=\emptyset$.
\begin{proposition}\label{Lexpl}\label{CexY}

a) Fix a $\mathfrak b$-dominant function $f\in\mathbb F^{\mathbb Z_{>0}}$ with $|f|<\infty$. We have \begin{eqnarray}\begin{aligned}\operatorname{I}_\mathfrak b(f)=\operatorname{Ann}_{\operatorname{U}(\mathfrak g(\infty))}( V^{Y_l(f)}
\otimes({\bf \Lambda}^\cdot)^{\otimes(q-p)}\otimes V_*^{Y_r(f)})\label{Eclsbcd}\end{aligned}\end{eqnarray}
in the $A$-case,
\begin{equation}\operatorname{I}_\mathfrak b(f)=\operatorname{Ann}_{\operatorname{U}(\mathfrak g(\infty))}(V^{Y_l(f)}\otimes ({\bf \Lambda}^\cdot)^{\otimes q})\label{F1}\end{equation}
in the $B/C/D$-cases whenever $f$ is integral, and
\begin{equation}\operatorname{I}_\mathfrak b(f)=\operatorname{Ann}_{\operatorname{U}(\mathfrak g(\infty))}(V^{Y_l(f)}\otimes ({\bf \Lambda}^\cdot)^{\otimes q-\frac12}\otimes \operatorname{Spin})\footnote{Taking into account the equality $\operatorname{Ann}_{\operatorname{U}(\mathfrak o(\infty))}(\operatorname{Spin}\otimes\operatorname{Spin})=\operatorname{Ann}_{\operatorname{U}(\mathfrak o(\infty))}({\bf \Lambda}^\cdot)$, and thus thinking of $\operatorname{Spin}$ as $({\bf \Lambda}^\cdot)^\frac12$, one sees the analogy between formulas~(\ref{F1}) and~(\ref{F2}).}\label{F2}\end{equation}in the $B/D$-cases whenever $f$ is half-integral.

b) A c.l.s. from~(\ref{Eclsbcd}) is of finite type.

c) Let $\mathfrak b^0$ be a fixed ideal Borel subalgebra of $\mathfrak g(\infty)$. Then any irreducible c.l.s. of finite type equals to $\operatorname{Q}_{\operatorname{L}_{\mathfrak b^0}(f^0))}$ for an appropriate $\mathfrak b^0$-dominant function $f^0\in\mathbb F^{\mathbb Z_{>0}}$.\end{proposition}
\begin{proof}The proof is entirely similar to the proof of~\cite[Proposition~2.10]{PP2}.\end{proof}
\begin{corollary}\label{C58}The set of annihilators of simple integrable highest weight modules coincides with the set of two-sided ideals of locally finite codimension in $\operatorname{U}(\mathfrak g(\infty))$.\end{corollary}
\subsection{Simple modules which are determined up to isomorphism by their annihilators}\label{SSnew}
It is well known that if $\mathfrak g$ is finite dimensional and semisimple, then a simple $\mathfrak g$-module $M$ is determined up to isomorphism by its annihilator in $\operatorname{U}(\mathfrak g)$ if and only if $M$ is finite dimensional. We now provide an analogue of this fact for $\mathfrak g(\infty)=\mathfrak{sl}(\infty), \mathfrak o(\infty), \mathfrak{sp}(\infty)$.

Recall that a simple object of the category $\mathbb T_{\mathfrak g(\infty)}$ is a simple $\mathfrak g(\infty)$-submodule of the tensor algebra T$^\cdot(V(\infty)\oplus V(\infty)_*)$ for $\mathfrak g(\infty)=\mathfrak{sl}(\infty)$, and of the tensor algebra T$^\cdot(V(\infty))$ for $\mathfrak g(\infty)=\mathfrak o(\infty), \mathfrak{sp}(\infty)$~\cite{DPS, PS}. It is easy to check that, for any fixed ideal Borel subalgebra $\mathfrak b^0$, the simple modules in the category $\mathbb T_{\mathfrak g(\infty)}$ are precisely the highest weight modules $\operatorname{L}_{\mathfrak b^0}(f)$ for which $f$ can be chosen to be integral and constant except at finitely many points (recall that the isomorphism class of a module $\operatorname{L}_{\mathfrak b^0}(f)$ recovers $f$ in the $B/C/D$-cases, and recovers $f$ up to an additive constant in the $A$-case).
We refer to these modules as {\it simple tensor modules}.
\begin{proposition}\label{P59} Let $M$ be a simple $\mathfrak{sl}(\infty)$-module which is determined up to isomorphism by its annihilator $I=\operatorname{Ann}_{\operatorname{U}(\mathfrak g(\infty))}M$. If $I$ is integrable, then $M$ is isomorphic to a simple tensor module.\end{proposition}
\begin{proof} If $I$ is not of locally finite codimension, then a straightforward analogue of~\cite[Lemma~6.8]{PP2} implies that there exist $f_1, f_2\in \mathbb F^{\mathbb Z_{>0}}$ such that $$\operatorname{Ann}_{\operatorname{U}(\mathfrak g(\infty))} \operatorname{L}_{\mathfrak b^0}(f_1)=\operatorname{Ann}_{\operatorname{U}(\mathfrak g(\infty))} \operatorname{L}_{\mathfrak b^0}(f_2)=I$$but $\operatorname{L}_{\mathfrak b^0}(f_1)\not\cong \operatorname{L}_{\mathfrak b^0}(f_2)$.

Assume now that $I$ has locally finite codimension. Then $I=\operatorname{I}(Q)$ for an irreducible c.l.s. of finite type $Q$, and by Proposition~\ref{Lexpl} c) $M$ is isomorphic to $\operatorname{L}_{\mathfrak b^0}(f^0)$ for some ideal Borel subalgebra $\mathfrak b^0$ and some $\mathfrak b^0$-dominant function $f^0$. Moreover, as $I$ is clearly fixed under the group $\tilde G:= \{g\in\operatorname{Aut}_\mathbb FV(\infty)\mid g^*(V(\infty)_*)=V(\infty)_*\}$ considered as a group of automorphisms of $\operatorname{U}(\mathfrak{sl}(\infty))$, it follows that $M$ is invariant under $\tilde G$. Now Theorems~3.4 and~4.2 in~\cite{DPS} imply that $\operatorname{L}_\mathfrak b(f)$ is a simple tensor module.

It remains to show that a simple tensor $\mathfrak g(\infty)$-module $M$ is determined up to isomorphism by its annihilator $\operatorname{Ann}_{\operatorname{U}(\mathfrak g(\infty))}M$. If $M'$ is a simple $\mathfrak g(\infty)$-module with $\operatorname{Ann}_{\operatorname{U}(\mathfrak g(\infty))}M'=\operatorname{Ann}_{\operatorname{U}(\mathfrak g(\infty))} M=I$, then the fact that $I$ has locally finite codimension implies that $M'$ is integrable and that the c.l.s. of $M'$ coincides with the c.l.s. of $I$, i.e., $\operatorname{Q}_M=\operatorname{Q}_{M'}$. A further consideration  (carried out in detail in A.~Sava's master's thesis~\cite{Sa}) shows that $M'$ is a highest weight $\mathfrak g(\infty)$-module with respect to the some ideal Borel subalgebra, and that the highest weight of $M$ equals the highest weight of $M'$. This implies $M'\cong M$.\end{proof}
\begin{remark} Any ideal $I\subset\operatorname{U}(\mathfrak g(\infty))$ as in Proposition~\ref{P59} has locally finite codimension. This follows from Corollary~\ref{C58} but also from the observation that the c.l.s. $\operatorname{Q}_M$ of a simple tensor module $M$ is of finite type. However, not every integrable highest weight module is a tensor module: this applies, for instance, to the integrable $\frak{sl}(\infty)$-module $L_{\frak b^0}(f)$ where $\frak b^0$ is an ideal Borel subalgebra corresponding to a partition $$\mathbb Z_{>0}=S_1\sqcup S_2\sqcup S_3,$$ and $f|_{S_1}=1, f|_{S_2}=f|_{S_3}=0$. Consequently, not every prime integrable ideal of locally finite codimension is the annihilator of a 
tensor module.\end{remark}

\section{Proofs of the results of Subsection~\ref{SStmr}}\label{S6}
In the present section $\mathfrak g(\infty)=\mathfrak o(\infty), \mathfrak{sp}(\infty)$. In Subsection~\ref{SSap} we prove a proposition which is an essential part of Theorem~\ref{1Tbcd}. The rest of the proofs we present in Subsection~\ref{Spr1tb}. They are relatively short but involve a lot of preliminary material from Subsections~\ref{SSsnot}, \ref{SScp}-\ref{SSSrrsa}.
\subsection{$S$-notation}\label{SSsnot}
We use the notation of Appendix~A. Let $S$ be a subset of $\mathbb Z_{>0}$. Put

$$\mathfrak g(S):=\begin{cases} \PenkovPetukhovBr{\{e^B_{\pm i, \pm j}\}_{i, j\in S}, \{e^B_{i, 0}, e^B_{0, i}\}_{i\in S}}&\mbox{in~the~$B$-case,}\\
\PenkovPetukhovBr{\{e^C_{\pm i, \pm j}\}_{i, j\in S}}&\mbox{in~the~$C$-case,}\\
\PenkovPetukhovBr{\{e^D_{\pm i, \pm j}\}_{i, j\in S}}&\mbox{in~the~$D$-case.}\\
\end{cases}$$
We have $\mathfrak g(\mathbb Z_{>0})=\mathfrak g(\infty)$.

Set $\mathfrak h_S:=\mathfrak h\cap\mathfrak g(S)$, and observe that $\mathfrak h_S=\PenkovPetukhovBr{e_{i,-i}}_{i\in S}$ in the $B/C/D$-cases. Note that

$\bullet$ if $S$ is finite, then $\mathfrak g(S)$ is isomorphic to $\mathfrak{sl}(n)$ in the $A$-case, to $\mathfrak o(2n+1)$ in the $B$-case, to $\mathfrak{sp}(2n)$ in the $C$-case, and to $\mathfrak o(2n)$, in the $D$-case, where $n=|S|$ is the cardinality of $S$; in addition, $\mathfrak h_S$ is a Cartan subalgebra of $\mathfrak g(S)$;

$\bullet$ if $S$ is infinite, then $\mathfrak g(S)$ is isomorphic to $\mathfrak g(\infty)$, and $\mathfrak h_S$  is a splitting Cartan subalgebra of $\mathfrak g(S)$.\\
Put also $\mathfrak b_S:=\mathfrak g(S)\cap\mathfrak b$ for the fixed splitting Borel subalgebra $\mathfrak b$ of $\mathfrak g(\infty)$. Clearly,

$\bullet$ if $S$ is finite, then $\mathfrak b_S$ is a Borel subalgebra of $\mathfrak g(S)$,

$\bullet$ if $S$ is infinite, then $\mathfrak b_S$ is a splitting Borel subalgebra of $\mathfrak g(S)$.

Let $\mathbb F^S$ denote the set of functions from $S$ to $\mathbb F$. Then $\mathbb F^S$ is a vector space of
dimension~$|S|$ if $S$ is finite. When $S=\{1,..., n\}$ we write simply $\mathbb F^n$ instead of $\mathbb F^{\{1,..., n\}}$. There is an isomorphism $\mathbb F^S\cong
\mathfrak h_S^*$ if $|S|>1$:\begin{equation}f\mapsto \lambda_f,\hspace{10pt}
\lambda_f(e_{i, -i})=f(i).\end{equation}
Next, we set $$\operatorname{M}_{\mathfrak b_S}(f):=\operatorname{U}(\mathfrak g(S))\otimes_{\operatorname{U}(\mathfrak b_S)}\mathbb F_{f}$$ for all $f\in\mathbb F^S$, where $\mathbb F_f$ is the 1-dimensional $\mathfrak b_S$-module assigned to $f$ as in Subsection~\ref{SStmr}. By $\operatorname{L}_{\mathfrak b_S}(f)$  we denote the unique simple quotient of $\operatorname{M}_{\mathfrak b_S}(f)$.

\subsection{Application of $S$-notation}\label{SSap} In this subsection we use $S$-notation to prove the following proposition. This proof is taken almost verbatim from the proof of~\cite[Proposition 4.1]{PP2}.
\begin{proposition}\label{Papp}Let $f\in\mathbb F^{\mathbb Z_{>0}}$. If $\operatorname{I}(f)\ne0$, then $|f|<\infty$.\end{proposition}

In the rest of this subsection we omit the superscripts $^{B/C/D}$ and write simply $\mathfrak g(\infty), \mathfrak g(S), \mathfrak g(n)$ instead of $$\mathfrak g^{B/C/D}(\infty),\hspace{10pt} \mathfrak g^{B/C/D}(S),\hspace{10pt} \mathfrak g^{B/C/D}(\{1,..., n\}).$$

The radical ideals of the center
$\operatorname{ZU}(\mathfrak g(n))$ of $\operatorname{U}(\mathfrak g(n))$ are in one-to-one
correspondence with $\mathfrak G_n$-invariant closed subvarieties of $\mathfrak
h_n^*$, where $\mathfrak h_n:=\mathfrak h\cap \mathfrak g(n)$ is a fixed Cartan subalgebra of $\mathfrak g(n)$ and $\mathfrak G_n$ is the respective Weyl group. Let
$I$ be an ideal of $\operatorname{U}(\mathfrak g(n))$. Then $\operatorname{ZVar}(I)$ denotes the subvariety of $\mathfrak h_n^*$
corresponding to the radical of the ideal $I\cap\operatorname{ZU}(\mathfrak g(n))$ of $\operatorname{ZU}(\mathfrak g(n))$. If
$\{I_t\}$ is any collection of ideals in $\operatorname{U}(\mathfrak g(n))$, then
\begin{equation}\operatorname{ZVar}(\cap_tI_t)=\overline{\cup_t\operatorname{ZVar}(I_t)}\label{Ezv},\end{equation} where, as in Subsection~\ref{SSavpi}, bar indicates Zariski closure.

Let $\phi: \{1,..., n\}\to \mathbb Z_{>0}$ be an injective map.
Slightly abusing notation, we denote 
by $\phi$ the induced homomorphism\begin{center}$\phi:
\operatorname{U}(\mathfrak g(n))\to\operatorname{U}(\mathfrak g(\infty))$.\end{center}
By $\operatorname{inj}(n)$ we denote the set of injective maps from $\{1,..., n\}$ to $\mathbb Z_{>0}$, and by $\operatorname{inj_0}(n)$ the set of order preserving maps from $\{1,..., n\}$ to $\mathbb Z_{>0}$ with respect to the standard order on $\{1,..., n\}$ and the order $\prec$ on $\mathbb Z_{>0}$.

For any $f\in \mathbb F^{\mathbb Z_{>0}}$ and $\phi\in\operatorname{inj}(n)$ we set $f_\phi:=f\circ\phi$. Then $\operatorname{M}(f_\phi):=\operatorname{M}_{\mathfrak b_{\operatorname{im}\phi}}(f_\phi)$ and $\operatorname{L}(f_\phi):=\operatorname{L}_{\mathfrak b_{\operatorname{im}\phi}}(f_\phi)$ are well defined $\mathfrak b_{\operatorname{im}\phi}$-highest weight $\mathfrak g(\phi)$-modules. If $f$ is $\mathfrak b$-dominant and $\phi\in\operatorname{inj}_0(n)$, then $f_\phi$ is $\mathfrak b_{\operatorname{im}\phi}$-dominant. 

Let $\phi\in\operatorname{inj}_0(n)$. By $\mathfrak g(\phi)$ we denote $\mathfrak g(\operatorname{im}\phi)\subset\mathfrak g(\infty)$. Let 
$\widetilde{\operatorname{M}}(f)$ be any quotient of $\operatorname{M}(f)$. It is well known
that\begin{center}$\operatorname{ZVar}(\operatorname{Ann}_{\operatorname{U}(\mathfrak g(\phi))}
\operatorname{M}(f_\phi))=\operatorname{ZVar}(\operatorname{Ann}_{\operatorname{U}(\mathfrak g(\phi))}\widetilde{\operatorname{M}}(f_\phi))=\operatorname{ZVar}(\operatorname{Ann}_{\operatorname{U}(\mathfrak g(\phi))}\operatorname{L}(f_\phi))=$\\$=\mathfrak
G_n(\rho_n+\lambda_{f_\phi})$,\end{center} where $\rho_n\in\mathfrak
h_n^*$ is the half-sum of positive roots.

Let $\mathfrak g$ be a Lie algebra. {\it The adjoint group of $\mathfrak g$} is the subgroup of $\operatorname{Aut}\mathfrak g$ generated by the exponents of all nilpotent elements of $\mathfrak g$. We denote this group by $\operatorname{Adj}\mathfrak g$.

\begin{lemma}\label{Lconj}Let $\phi_1: \mathfrak k\to\mathfrak g$ and $\phi_2: \mathfrak k\to\mathfrak g$ be two $\operatorname{Adj}\mathfrak g$-conjugate morphisms of Lie
algebras. Let $I$ be a two-sided ideal of $\operatorname{U}(\mathfrak g)$. Then
\begin{center}$\phi_1^{-1}(I)=\phi_2^{-1}(I)$.\end{center}\end{lemma}
\begin{proof}The adjoint action of $\mathfrak g$ on $\operatorname{U}(\mathfrak g)$ extends
uniquely to an action of $\operatorname{Adj}\mathfrak g$ on $\operatorname{U}(\mathfrak g)$. The ideal
$I$ is $\mathfrak g$-stable and thus is $\operatorname{Adj}\mathfrak g$-stable. Let $g\in
\operatorname{Adj}\mathfrak g$ be such that $\phi_1=g\circ\phi_2$.
Then\begin{center}$\phi_1^{-1}(g(i))=\phi_2^{-1}(i)$\end{center}for any $i\in I$. Hence,
\begin{center}$\phi_1^{-1}(I)=\phi_2^{-1}(I)$.\end{center}\end{proof}

\begin{proof}[Proof of Proposition~\ref{Papp}]Let $\operatorname{I}(f)\ne 0$. Assume to the contrary that there exist $i_1,...,i_s,...\in
\mathbb Z_{>0}$ such that $$f(i_1),...., f(i_s),...$$are pairwise distinct
elements of $\mathbb F$. As $\operatorname{I}(f)\ne0$, there exists a positive integer $n$ and an injective
map $\phi: \{1,..., n\}\to \mathbb Z_{>0}$ such
that\begin{center}$I_\phi:=\operatorname{I}(f)\cap\operatorname{U}(\mathfrak g(\phi))\ne
0$,\end{center} or equivalently
\begin{equation}\operatorname{U}(\mathfrak g(n))\supset \phi^{-1}(\operatorname{I}(f))=\phi^{-1}(I_\phi)\ne0.\label{Eslp}\end{equation}

Let $\psi\in\operatorname{inj}(n)$ be another map. Since $\phi$ and $\psi$ are conjugate via the adjoint group of $\mathfrak g(\infty)$, we have
\begin{equation}\phi^{-1}(\operatorname{I}(f))=\psi^{-1}(\operatorname{I}(f))\ne0\label{Econj}.\end{equation}
This means that $\phi^{-1}(\operatorname{I}(f))$ depends on $n$ and $f$ but not on $\phi$, and we set $$I_n:=\phi^{-1}(\operatorname{I}(f)).$$

Assume now that $\phi\in\operatorname{inj_0}(n)$. Then the
highest weight space of the $\mathfrak g(\infty)$-module $\operatorname{L}_\mathfrak b(f)$ generates a highest weight $\mathfrak g(\phi)$-submodule
$\widehat{\operatorname{L}}(f_\phi)$. Clearly,
\begin{center}$\operatorname{Ann}_{\operatorname{U}(\mathfrak g(\phi))}\operatorname{L}_\mathfrak b(f)\subset
\operatorname{Ann}_{\operatorname{U}(\mathfrak g(\phi))}\widehat{\operatorname{L}}(f_\phi).$\end{center}Therefore,\begin{center}$I_n\subset\cap_{\phi\in
\operatorname{inj_0}(n)}\operatorname{Ann}_{\operatorname{U}(\mathfrak g(n))}\widehat{\operatorname{L}}(f_\phi)$\end{center}
and\begin{center}$I_n\cap\operatorname{ZU}(\mathfrak g(n))\subset\cap_{\phi\in
\operatorname{inj_0}(n)}(\operatorname{Ann}_{\operatorname{U}(\mathfrak g(n))} \widehat{\operatorname{L}}(f_\phi)\cap \operatorname{ZU}(\mathfrak g(n))).$\end{center}Hence,
according to~(\ref{Ezv}) we have
\begin{center}$\overline{\cup_{\phi\in \operatorname{inj_0}(n)}\mathfrak
G_n(\rho_n+\lambda_{f_\phi})}=\mathfrak G_n(\rho_n+\overline{\cup_{\phi\in
\operatorname{inj_0}(n)}\lambda_{f_\phi}})\subset \operatorname{ZVar}(I_n)$.\end{center}

We claim that
\begin{center}$\overline{\mathfrak
G_n(\cup_{\phi\in\operatorname{inj_0}(n)}\lambda_{f_\phi}})=\mathfrak
h_n^*,$
\end{center}and thus
that\begin{equation}\operatorname{ZVar}(I_n)=\mathfrak
h_n^*.\label{Efd1}\end{equation} Our claim is
equivalent to the equality \begin{center}$\overline{\mathfrak
G_n(\cup_{\phi\in\operatorname{inj_0}(n)}\lambda_{f_\phi}})=\overline{(\cup_{\phi\in\operatorname{inj}(n)}\lambda_{f_\phi}})=\mathfrak
h_n^*$
\end{center} which is
implied by the following
equality:\begin{equation}\overline{(\cup_{\phi\in\operatorname{inj}(n)}f_\phi})=\mathbb F^n\label{Efd}.\end{equation}

We now prove~(\ref{Efd}) by
induction. The inclusion $\{1,..., n-j\}\subset\{1,..., n\}$ induces a restriction map $$res: \mathbb F^n\to \mathbb F^{n-j}.$$ Denote by $f_\psi*$ the preimage of $f_\psi$ under
$res$ for $\psi\in\operatorname{inj}(n-j)$. 
We will show that 
\begin{equation}f_\psi*\subset\overline{\cup_{\phi\in
\operatorname{inj}(n)}f_\phi}\label{Efps}\end{equation}for any  $j\le n$ and any map $\psi\in\operatorname{inj}(n-j)$. This holds trivially for $j=0$. Assume that it also holds for $j$.
Fix $\psi\in \operatorname{inj}(n-j-1)$ and set
$$(\psi\times k)(l):=\begin{cases}\psi(l)&\mbox{if}~l\le n-j-1\\i_k&\mbox{if}~l=n-j\end{cases}.$$
It is clear that there exists $s\in\mathbb Z_{\ge1}$ such that
$$(\psi\times k)\in\operatorname{inj}(n-j)$$ for any $k\in\mathbb Z_{\ge s}$.
Moreover, $f_{\psi\times k_1}\ne f_{\psi\times k_2}$ for any $k_1\ne k_2$.
Therefore$$\overline{\cup_{k\in\mathbb Z_{\ge s}}f_{\psi\times
k}*}=f_\psi*,$$which yields~(\ref{Efps}).

For $j=n$,~(\ref{Efps}) yields $\mathbb F^n\subset \overline{\cup_{\phi\in
\operatorname{inj}(n)}f_\phi}$, consequently~(\ref{Efd}) holds. Then~(\ref{Efd1}) holds also, hence


$$I_n\cap\operatorname{ZU}(\mathfrak g(n))=0.$$
It is a well known fact that an ideal of $\operatorname{U}(\mathfrak g(n))$ whose intersection with $\operatorname{ZU}(\mathfrak g(n))$ equals zero is the zero ideal~\cite[Proposition 4.2.2]{Dix}. Therefore, we have a contradiction with~(\ref{Eslp}), and the proof is complete.
\end{proof}

\subsection{Combinatorics of partitions}\label{SScp}
\subsubsection{Partitions} In this paper, by a {\it partition} $p$ we understand a nondecreasing sequence $$p(0)\le p(1)\le...\le p(m)$$ of positive integers. We set $|p|:=p(0)+p(1)+...+p(m)$ and $\sharp p:=m+1$.
Clearly, any finite sequence of nonnegative integers $p(0),..., p(m)$ defines a unique partition via reordering and deleting possible zeros. 

For a partition $p=\{p(0),..., p(m)\}$, put $\widehat p(i):=|\{j\mid p_j\ge i\}|$. Let $\widehat p$ be the conjugate partition, i.e. the partition defined by the sequence $\widehat p(1), \widehat p(2),...$. Two partitions $p'$ and $p''$ can be combined into the partition $p'+p''$ obtained by reordering the sequence $p'(0), p'(1),..., p''(0), p''(1),...$. We set also $p'\widehat+p'':=\widehat q$ where $q=\widehat p'+\widehat p''$.

Given a partition $p=\{p(0),..., p(m)\}$, consider the sequence $$p^*(0),..., p^*(m-1), p^*(m),..., p^*(2m)$$ where $$p^*(0)=...=p^*(m-1)=0, p^*(m)=p(0),..., p^*(2m)=p(m).$$
Let $p^e$ be the partition corresponding to the sequence $p^*(0)\le p^*(2)\le...\le p^*(2m)$, and $p^o$ be the partition corresponding to the sequence $p^*(1)\le...\le p^*(2m-1)$ (``e'' and ``o'' stand for ``even'' and ``odd'').
\subsubsection{Lusztig sequences}
Let $Z_m$ denote the set of subsets of nonnegative integers with $m+1$ elements. An element $z\in Z_m$  is represented by a sequence $0\le z_0<z_1<...<z_m$. We assign to such a sequence $z\in Z_m$ the sequence $$p(z)(i):=z_i-i,$$ and denote by $p(z)$ the partition corresponding to this sequence. Conversely, to a partition $p$ with $\sharp p\le m+1$ we attach the sequence $z(p)\in Z_m$ such that $p(z(p))=p$. We say that two sequences $z\in Z_m$ and $z'\in Z_{m'}$ are {\it equivalent} if they correspond to the same partition.

For $z\in Z_{2m}$ we denote by $z^{even}$ the subsequence of $z$ which consists of even elements; we put also $z^{odd}:=z\backslash z^{even}$. We renumber the subsequences $z^{even}$ and $z^{odd}$ in the obvious way, and set  $$z_i^{Le}:=\frac{z^{even}_i}2, z_i^{Lo}:=\frac{z^{odd}_i-1}2.$$(``$L$'' stands for Lusztig). In addition, for a partition $p$ we set $$p^{Le}:=p(z(p)^{Le}), p^{Lo}:=p(z(p)^{Lo}).$$

In the rest of this subsection we will frequently work with another nonnegative integer $m'$. We put $$\Delta m:=m-m'.$$
Next, we define a partial inverse to the map $p\to (p^{Le}, p^{Lo})$. For $z\in Z_m, z'\in Z_{m'}$ we set $p:=p(z), p':=p(z')$. Let
$$\langle z, z'\rangle_{\Delta m} :=\{2z_0,..., 2z_m\}\sqcup\{2z'_0+1,..., 2z_{m'}'+1\}.$$

Clearly $\langle z, z'\rangle_{\Delta m}$ is a subset of positive integers with $m+m'+2$ elements and thus $\langle z, z'\rangle_{\Delta m}\in Z_{m+m'+1}$. Define $$\langle p, p'\rangle_{\Delta m}:=p(\langle z, z'\rangle_{\Delta m}).$$ It is easy to see that$$\langle p, p'\rangle_{\Delta m}^{Le}=p~\mbox{~and~}\langle p, p'\rangle_{\Delta m}^{Lo}=p'.$$ We say that $p$ is a {\it BV-partition} if $p=\langle p^{Le}, p^{Lo}\rangle_1$ (``$BV$'' stands for Barbasch and Vogan).
\subsubsection{Barbasch-Vogan functions} Consider two partitions $p', p''$. Given (unique) $z'\in Z_m, z''\in Z_{m'}$ such that $p'=p(z')$ and $p''=p(z'')$, we can consider $z'$ and $z''$ as partitions. This allows us to define $(z'+z'')^e$ and $(z'+z'')^o$. It is clear that all elements of $(z'+z'')^e$ are distinct, and thus $(z'+z'')^e\in Z_{m_e}$ where $m_e:=\lceil\frac{m+m'}2\rceil$ ($\lceil\cdot\rceil$ stands for ceiling). Similarly $(z'+z'')^o\in Z_{m_o}$ for $m_o:=\lfloor\frac{m+m'}2\rfloor$ ($\lfloor\cdot\rfloor$ stands for floor). We put
$$p~\star^e_{\Delta m}p':=p((z+z')^e),\hspace{10pt} p\star^o_{\Delta m}p':=p((z+z')^o).$$


The following functions play a significant role in what follows:

$$\mbox B(p_1, p_2, p_3):=\left<(p_1^{Le}\star_{1}^op_1^{Lo}) \widehat + (p_2^{Le}\star_{1}^op_2^{Lo})\widehat + p_3^o,  (p_1^{Le}\star_{1}^ep_1^{Lo})\widehat +(p_2^{Le}\star_{1}^ep_2^{Lo})\widehat +p_3^e\right>_{{-1}},$$

$$\mbox C(p_1, p_2, p_3):=\left<(p_1^{Le}\star_{1}^ep_1^{Lo}) \widehat + (p_2^{Le}\star_{0}^ep_2^{Lo})\widehat + p_3^e,  (p_1^{Le}\star_{1}^op_1^{Lo})\widehat +(p_2^{Le}\star_{0}^op_2^{Lo})\widehat +p_3^o\right>_{{1}},$$

$$\mbox D(p_1, p_2, p_3):=\left<(p_1^{Le}\star_{0}^op_1^{Lo}) \widehat + (p_2^{Le}\star_{0}^op_2^{Lo})\widehat + p_3^e,  (p_1^{Le}\star_{0}^ep_1^{Lo})\widehat +(p_2^{Le}\star_{0}^ep_2^{Lo})\widehat +p_3^o\right>_{{0}}.$$
\subsubsection{Equalities}
We have $p=\langle p^{Le}, p^{Lo}\rangle_{\Delta m}$ for some integer $\Delta m$.
 It is known that $|\widehat p|=|p|$ and $\widehat{\widehat p}=p$. We have
$$\begin{array}{|c|c|}\hline
1&\sharp p=\max(2\sharp p^e-1, 2\sharp p^o)\\\hline
2&\sharp \langle p, p'\rangle_{\Delta m}=\max(2\sharp p-\Delta m,2\sharp p'+\Delta m-1)\\\hline
3&\max(\sharp p, \sharp p'+\Delta m)=\max (\sharp p\star^e_{\Delta m}p'+\lfloor\frac{\Delta m}2\rfloor, \sharp p\star_{\Delta m}^op'+\lceil\frac{\Delta m}2\rceil)\\\hline
4&\sharp(p'\widehat{+}p''):=\max(\sharp p', \sharp p'')\\
\hline\end{array}$$

\subsubsection{Inequalities}\label{SSSin} Assume that $a, b, \Delta a, \Delta b$ are integers. Then we have

$$|\max (a+\Delta a, b+\Delta b)-\max (a, b)|\le\max(|\Delta a|, |\Delta b|).$$
This implies

$$\begin{array}{|c|llc|}\hline
\mbox{Label}&\mbox{Inequality}&&\\\hline
1&|\sharp p -\max(2\sharp p^e, 2\sharp p^o)|&\le& 1\\\hline
2&|\sharp \langle p, p'\rangle_{\Delta m}-\max(2\sharp p,2\sharp p')|&\le&\max(|\Delta m-1|, |\Delta m|)\\\hline

3.1&|\max(\sharp p, \sharp p'+\Delta m)-\max(\sharp p, \sharp p')|&\le&|\Delta m|\\\hline

3.2&\begin{array}{c}|\max (\sharp p\star^e_{\Delta m}p'+\lfloor\frac{\Delta m}2\rfloor, \sharp p\star_{\Delta m}^op'+\lceil\frac{\Delta m}2\rceil)-\\-\max (\sharp p\star^e_{\Delta m}p', \sharp p\star_{\Delta m}^op')|\end{array}&\le&\max(|\lfloor\frac{\Delta m}2\rfloor|, |\lceil\frac{\Delta m}2\rceil|)\\\hline

\end{array}$$
Under the assumption that $p, p_1$ and $p_2$ are Barbasch-Vogan partitions, and $p_3$ is an arbitrary partition, we obtain
$$\begin{array}{|c|ccccc|}\hline
\mbox{Label}&\mbox{Inequality}&&&&\\\hline
BC1&|\sharp p &-& 2\max(\sharp p^{Le}\star_{1}^op^{Lo}, \sharp p^{Le}\star_{1}^ep^{Lo})|&\le& 1+1+1\\\hline
D1&|\sharp p &-& 2\max(\sharp p^{Le}\star_{0}^ep^{Lo}, \sharp p^{Le}\star_{0}^op^{Lo})|&\le& 1+0+0\\\hline
B2&|\sharp B(p_1, p_2, p_3) &-& \max(\sharp p_1, \sharp p_2, \sharp p_3)|&\le& 3+2\\\hline
C2&|\sharp C(p_1, p_2, p_3) &-& \max(\sharp p_1, \sharp p_2, \sharp p_3)|&\le& 3+1\\\hline
D2&|\sharp D(p_1, p_2, p_3) &-& \max(\sharp p_1, \sharp p_2, \sharp p_3)|&\le& 1+1\\\hline
\end{array}.$$
\subsection{The associated variety of a simple $\mathfrak g$-highest weight module}\label{SSasv}
Let $\mathfrak g$ be a finite-dimensional semisimple Lie algebra with Borel subalgebra $\mathfrak b\subset \mathfrak g$ and Cartan subalgebra $\mathfrak h\subset\mathfrak b$. Let $\Delta\subset\mathfrak h^*$ be the root system of $\mathfrak g$, $W$ be the corresponding Weyl group and $\Delta^\pm$ be the set of positive and negative roots. By $\rho$ we denote the half-sum of all positive roots, and for any $\lambda\in\mathfrak h^*$ we denote by $\operatorname{L}(\lambda)$ the simple $\mathfrak g$-module with $\mathfrak b$-highest weight $\lambda$.

According to Duflo's Theorem, any primitive ideal of
$\operatorname{U}(\mathfrak g)$ is the annihilator of $\operatorname{L}(\lambda)$ for some $\lambda\in\mathfrak h^*$. The associated variety of $\operatorname{Ann}_{\operatorname{U}(\mathfrak g)}\operatorname{L}(\lambda)$ is the closure of a certain nilpotent coadjoint orbit $\EuScript O(\lambda)$ of $\mathfrak g^*\cong\mathfrak g$~\cite{Jo}.

From now on we fix $\lambda$. The goal of this and the next two subsections is to provide an explicit way for computing $\EuScript O(\lambda)$ when $\mathfrak g$ is a simple classical Lie algebra or a direct sum of simple classical Lie algebras. We first 
consider the case of regular integral weight $\lambda$ and then explain how to handle the general case modulo some computation in the category of finite groups which is carried out in~\cite{Lusztig}.

Assume first that $\mathfrak g$ is simple. Fix an invariant nondegenerate scalar product $(\cdot, \cdot)$ on $\mathfrak g$. The restriction of $(\cdot, \cdot)$ on $\mathfrak h$ is also nondegenerate and hence defines a scalar product $(\cdot, \cdot)$ on $\mathfrak h^*$. We recall that a weight $\mu$ is {\it regular} if $(\mu, \alpha)\ne 0$ for all $\alpha\in\Delta$, and that $\mu$ is {\it integral} if $\frac{2(\mu, \alpha)}{(\alpha, \alpha)}\in\mathbb Z$ for any $\alpha\in\Delta$.

Assume that $\lambda+\rho$ is regular and integral. Then there exists a unique $w\in W$ such that $w^{-1}(\lambda+\rho)-\rho$ is dominant. It is well known that in this case $$\EuScript O(\lambda)=\EuScript O(w\rho-\rho).$$ Thus we have a ``commutative diagram''

$$\begin{array}{ccc}\lambda&\mapsto&\EuScript O(\lambda)\\
\begin{turn}{270}$\mapsto$\end{turn}&&\begin{turn}{270}=\end{turn}\\
w&\mapsto&\EuScript O(w\rho-\rho)\end{array}.$$
The map $w\mapsto\EuScript O(w\rho-\rho)$ is described in~\cite{Barbasch-Vogan} for all classical simple Lie algebras.

Assume that $\lambda$ is regular but not necessarily integral. We set $$\Delta(\lambda):=\{\alpha\in\Delta\mid \frac{2(\alpha, \lambda)}{(\alpha, \alpha)}\in\mathbb Z\}.$$ It is clear that $\Delta(\lambda)$ is a root system and thus corresponds to a Lie algebra $\mathfrak g(\lambda)$~(which is not necessarily a subalgebra of $\mathfrak g$). If $\mathfrak g$ is classical, then $\Delta(\lambda)$ is a direct sum of simple root systems of classical type. We denote by $W(\lambda)$ the reflection subgroup of $W$ generated by $\Delta(\lambda)$, and refer to it as the {\it integral Weyl group} of $\lambda$ (note that if $\lambda$ is integral then $W(\lambda)=W$). As in the previous case, there exists a unique $w\in W(\lambda)$ such that $w^{-1}(\lambda+\rho)-\rho$ is {\it dominant}, i.e. such that $w^{-1}(\lambda+\rho)-w'(\lambda+\rho)$ is a sum of negative roots from $\Delta(\lambda)$ for any $w'\in W(\lambda)$. It is well known that in general $\EuScript O(\lambda)\ne\EuScript O(w\rho-\rho)$, but nevertheless one can compute $\EuScript O(\lambda)$ for a given triple $(w, W(\lambda), W)$.

To proceed further we need the notion of Springer correspondence. Namely,  one can attach to a nilpotent coadjoint orbit of $\EuScript O\subset\mathfrak g^*$ a simple module $\operatorname{Spr}(\EuScript O)$ over the Weyl group $W$ of $\mathfrak g$~\cite[Section 10.1]{CM}. This correspondence is injective~\cite[Section 10.1]{CM}, and for a simple module $E$ of $W$ we denote by $\EuScript O(E)$ the nilpotent coadjoint orbit $\EuScript O\subset\mathfrak g^*$ for which $\operatorname{Spr}(\EuScript O)=E$. Note that $\EuScript O(E)$ may not exist. We set $$\operatorname{Spr}_{\Delta}(w):=\operatorname{Spr}(\EuScript O(w\rho-\rho))$$where the subscript $\Delta$ keeps track of the Weyl group of which $w$ is an element.
The map $w\mapsto \operatorname{Spr}_{\Delta}(w)$ is essentially a combinatorial object, and one should be able to provide a combinatorial description of this map. In the case when $\mathfrak g$ is a classical finite-dimensional Lie algebra, this is done in~\cite{Barbasch-Vogan}.

We will use the following notation $$\operatorname{Irr}(W), \operatorname{Irr}(W(\lambda)), \operatorname{Irr}(W)^\dagger, \operatorname{Irr}(W(\lambda))^\dagger, a_E, b_E, j^W_{W(\lambda)}(E)$$
of~\cite{Lusztig} (note that $W(\lambda)$ is always a parahoric subgroup of $W$ and that this fact is needed to properly define $j^W_{W(\lambda)}(E)$).

The equality $\EuScript O(\lambda)=\EuScript O(j^W_{W(\lambda)}(\operatorname{Spr}_{\Delta(\lambda)}(w)))$ is a consequence of results of~\cite{Jo}, see also~\cite[Subection 7.6]{LO}.

Here is how to reduce the case of nonregular $\lambda+\rho$ to the regular case. Namely, assume that $\lambda+\rho$ is not regular. We say that $\lambda'+\rho$ is a {\it regularization} of $\lambda+\rho$ if $\lambda'+\rho$ is regular and

1) if $(\lambda+\rho, \alpha)\in\mathbb Z_{>0}$ then $(\lambda'+\rho, \alpha)\in\mathbb Z_{>0}$ for all $\alpha\in\Delta^+$,

2) if $(\lambda+\rho, \alpha)\in\mathbb Z_{\le 0}$ then $(\lambda'+\rho, \alpha)\in\mathbb Z_{<0}$ for all $\alpha\in\Delta^+$,

3) if $(\lambda+\rho, \alpha)\not\in\mathbb Z$ then $(\lambda'+\rho, \alpha)\not\in\mathbb Z$ for all $\alpha\in\Delta^+$.\\
If $\lambda'+\rho$ is a regularization of $\lambda+\rho,$ then $\EuScript O(\lambda')=\EuScript O(\lambda)$. Such a regularization always exists. For example if $N>\!\!>0$ then  $$\lambda'+\rho:=N(\lambda+\rho)+\rho$$ is a regularization of $\lambda+\rho$.

Finally, we reduce the case of semisimple Lie algebra $\mathfrak g$ to the case of simple Lie algebra $\mathfrak g$. Namely, we fix a decomposition $\mathfrak g=\oplus_i\mathfrak g_i$ for simple ideals $\mathfrak g_i$. Then $\EuScript O(\lambda)=\oplus_i\EuScript O(\lambda_i)$, where $\lambda_i$ is the orthogonal projection of $\lambda$ to $(\mathfrak h\cap\mathfrak g_i)^*$.

\subsection{Discussion of the algorithm}
In the next subsection we provide an explicit combinatorial algorithm which computes $\EuScript O(\lambda)$ for any weight $\lambda$ of a simple classical Lie algebra. We use the notation of Subsection~\ref{SSasv} and set $$\EuScript O^{B/C/D}(f):=\EuScript O(\lambda_f)$$ for the Lie algebras $\mathfrak{so}(2n+1), \mathfrak{sp}(2n), \mathfrak{so}(2n)$ respectively. In these cases one can attach to any nilpotent coadjoint orbit a partition formed by the sizes of Jordan blocks of any element $x$ in the orbit, where $x$ is considered as a linear operator, see~\cite{CM}. More precisely, given $\EuScript O^{B/C/D}(f)\subset\mathfrak g^{B/C/D}(n)$ we denote the above partition by $p^{B/C/D}(f)$. Clearly, $|p^B(f)|=2n+1$ and $|p^{C/D}(f)|=2n$.

The simple modules of $W=W^{B/C/D}(n)$ are parametrized by pairs of partitions~\cite{Barbasch-Vogan}. In~\cite{Cr} one can find a description of the Springer correspondence at the level of partitions, see also~\cite{Barbasch-Vogan}. In our notation this correspondence can be written in the following way:
$$\begin{array}{ccc}\operatorname{Spr}(f)\leftrightarrow&(p^B(f)^o, p^B(f)^e)&\mbox{in the $B$-case},\\
\operatorname{Spr}(f)\leftrightarrow&(p^C(f)^e, p^C(f)^o)&\mbox{in the $C$-case},\\
\operatorname{Spr}(f)\leftrightarrow&(p^D(f)^o, p^D(f)^e)&\mbox{in the $D$-case},\\
\end{array}$$
see~\cite[p.~165]{Barbasch-Vogan}. Moreover, if $\operatorname{Spr}(f)\leftrightarrow(\alpha, \beta)$ then
$$\begin{array}{ccc}p^B(f)=&\langle\beta, \alpha\rangle_{-1},\\
p^C(f)=&\langle\alpha, \beta\rangle_{1},\\
p^D(f)=&\langle\beta, \alpha\rangle_{0}.\end{array}$$

We set $\Delta(f):=\Delta(\lambda_f), W(f):=W(\lambda_f)$. The following should be considered as the scheme of the algorithm we aim at.

Find $W(f)$ and decompose it as the direct product $W_1\times W_2\times...$ of Weyl groups of simple root systems. Fix a regularization $\lambda'$ of $\lambda_f$. Find the unique element $w\in W(f)$ such that $$w^{-1}(\lambda'+\rho)-w'(\lambda'+\rho)$$ is a sum of negative roots from $\Delta(f)$ for any $w'\in W(f)$. Record $w$ as $$(w_1, w_2,...)\in W_1\times W_2\times....$$ Attach to $w_i$ the $W_i$-module $\operatorname{Spr}_{\Delta_i}(w_i)$. Compute $$E(f):=j_{W(f)}^W(E_{int}(f)),$$ where $E_{int}(f):=\otimes_i \operatorname{Spr}_{\Delta_i}(w_i)$. Then $\EuScript O(f)=\EuScript O(E(f))$. Denote the partition assigned to $\EuScript O(f)$ by $\operatorname{RS}^{B/C/D}_L(f)$.

This scheme translates into the following mnemonic algorithm.

Step 1(mnemonic). Add $\rho$ to $\lambda_f$.

Step 2(mnemonic). Determine the factors $W_1, W_2,...$ of $W(f)$ arising from the simple components of the root system $\Delta(f)$.

Step 3(mnemonic). Find a regularization $\lambda'+\rho$ of $\lambda_f+\rho$ and the element $w=(w_1, w_2,...)\in W(f)$ corresponding to $\lambda'$.

Step 4(mnemonic). To each $w_i$, assign a partition (in the $A$-case) or a pair of partitions (in the $B/C/D$-cases) as it is done in~\cite[Proposition 17]{Barbasch-Vogan}.

Step 5(mnemonic). Note that the datum assigned to $w_i$ in Step 4 corresponds naturally to the simple $W_i$-module $E_i:=\operatorname{Spr}_{\Delta_i}(w_i)$. Then, using~\cite{Lusztig}, compute the pair of partitions corresponding to the $W$-module $E(f):=j^W_{W(f)}(E_{int}(f))$. Finally, compute $\operatorname{RS}^{B/C/D}_L(f)=\EuScript O(E(f))$ using the Springer correspondence.

\subsection{The algorithm for $\mathfrak g^{B/C/D}(n)$}\label{SSrsa}

We now describe the precise  algorithm which computes $p^{B/C/D}(f)$. This is a compilation of several works~\cite{Jo, Lusztig, Barbasch-Vogan}.
Let $f\in\mathbb F^n$.


Step 1. Set
\begin{center}$\begin{tabular}{ll}$f^+:= (f(1)+\frac{2n-1}2, f(2)+\frac{2n-3}2,..., f(n)+\frac12)$& for the $B$-case,\\
$f^+:= (f(1)+n, f(2)+(n-1),..., f(n)+1)$& for the $C$-case,\\
$f^+:= (f(1)+(n-1), f(2)+(n-2),..., f(n)+0)$& for the $D$-case.\end{tabular}$\end{center}
Define the function $f^+: \{\pm1,...,\pm n\}\to\mathbb F$ by setting $f^+(-i):=-f^+(i)$ for\\ $i\in\{1,..., n\}.$

Step 2. Consider the set $\{1,..., n, -n,..., -1\}$ with linear order
\begin{equation}1\prec 2\prec 3\prec....\prec (n-1)\prec n\prec -n\prec -(n-1)\prec...\prec-3\prec-2\prec -1.\label{ES1n}\end{equation}
Put $f(-i):=-f(i)$ for $i\in\{1,..., n\}$, and introduce an equivalence relation $\sim$ on $\{\pm1,..., \pm n\}$:
\begin{center}$i\sim j$ if and only if $f(i)-f(j)\in\mathbb Z$.\end{center}
Denote the equivalence classes by $[\sim]_i$, and let $-[\sim]_i$ be the class with all signs reversed. 
Next, relabel the equivalence classes $[\sim]_i$ so that $[\sim]_1=\{i\mid f(i)\in\mathbb Z\}$, $[\sim]_2=\{i\mid f(i)\in\mathbb Z+\frac12\}$, and the equality $[\sim]_{2i+1}=-[\sim]_{2i+2}$ holds for $i\ge1$.
Let $n_i$ be cardinality of the equivalence class $[\sim]_i$\footnote{Note that $n_1, n_2$ might be equal to 0, and that $t$ is necessarily even.}.


Step 3. For every $i$ we introduce the following linear order $\triangleright$ on $[\sim]_i$. For $m\prec k\in [\sim]_i$, we 
set

(1) $m\triangleright k$ if and only if $f^+(m)-f^+(k)\in\mathbb Z_{>0}$,

(2) $k\triangleright m$ if and only if $f^+(m)-f^+(k)\in\mathbb Z_{\le0}$.

If we are in the $C$-case and $i=2$ or we are in the $D$-case and $i=1, 2$, we further modify the order $\triangleright$ as follows.
Consider the smallest possible value $v$ of $|f^+|$ on $[\sim]_i$, together with 
its preimage $$|f^+|^{-1}(v):=\{x\in [\sim]_i\mid |f^+(x)|=v\}$$ in $[\sim]_i$. For the $\prec$-maximal element $m$ of this preimage define $m\triangleright -m$. One can check that this yields a well-defined linear order on $\triangleright$ on the equivalence class $[\sim]_i$.




Step 4. Consider each equivalence class $[\sim]_i$ as a subsequence of~(\ref{ES1n}) together with the linear order $\triangleright$ from Step 4, and apply the Robinson-Schensted algorithm to $[\sim]_i$. The output is a pair of semistandard tableaux of the same shape, and this shape determines a partition $p_i$ of $n_i$ ($i\le t$).

Step 5. Set $\operatorname{RS}^{B/C/D}(f):=($B/C/D$)(p_1, p_2, p_{\ge3})$, where $$p_{\ge3}:=p_4\widehat+p_6\widehat+...\widehat+p_t.$$

\begin{proposition}\label{Ppart} Let $f\in\mathbb F^n$ be a function. Then $p^{B/C/D}(f)=\operatorname{RS}^{B/C/D}(f)$.\end{proposition}
\begin{proof} The left-hand side of our asserted equality is computed by the mnemonic algorithm described above. The right-hand side is computed by the combinatorial algorithm following the mnemonic algorithm. Therefore it is enough to compare the two algorithms. To start, we need an explicit description of the Weyl group $W^{B/C/D}(n)$. Namely, we identify $W:=W^{B/C/D}(n)$ with a subgroup of the group $Perm(\{\pm \varepsilon_1, \pm \varepsilon_2,..., \pm \varepsilon_n\})$ of permutations of $\{\pm \varepsilon_1, \pm \varepsilon_2,..., \pm \varepsilon_n\}$. More precisely, $W^B(n)$ and $W^C(n)$ are identified with the subgroup\begin{center}$\{w\in Perm(\{\pm\varepsilon_1,...,\pm \varepsilon_n\})\mid \forall i\exists j: w\{\varepsilon_i, -\varepsilon_i\}=\{\varepsilon_j, -\varepsilon_j\}\}$,\end{center} and $W^D(n)$ is identified with the subgroup of even permutations in $$W^C(n)=W^B(n).$$

Next we describe the integral Weyl group $W(f)$ of $\lambda_f$ in terms of the equivalence classes $[\sim]_i$ of Step 2. Namely, we have
$$\begin{tabular}{l}$W^B(f)\cong W^B(\frac{n_1}2)\times W^B(\frac{n_2}2)\times S_{n_4}\times S_{n_6}...\times S_{n_t}$,\\
$W^C(f)\cong W^C(\frac{n_1}2)\times W^D(\frac{n_2}2)\times S_{n_4}\times S_{n_6}...\times S_{n_t}$,\\
$W^D(f)\cong W^D(\frac{n_1}2)\times W^D(\frac{n_2}2)\times S_{n_4}\times S_{n_6}...\times S_{n_t}$,\end{tabular}$$
where

$\bullet$ the first factor is the subgroup of $W^{B/C/D}(n)$ which  keeps $[\sim]_i$ pointwise fixed for $i\ne 1$,

$\bullet$ the second factor is the subgroup [of even permutations in the $C$-case] of $W^{B/C/D}(n)$ which  keeps $[\sim]_i$ pointwise fixed for $i\ne 2$,

$\bullet$ the $i$-th factor ($i\ge3$) is the subgroup of $W^{B/C/D}(n)$ which keeps $[\sim]_j$ pointwise fixed for $j\ne 2i+1, 2i+2$.\\
Since the integers $n_1, n_2, ...$ are computed in Step 2 of the combinatorial algorithm, we see that Step 1 and Step 2 of both algorithms match each other.

According to the mnemonic algorithm, next we have to find a regularization $(\lambda'+\rho)$ of $(\lambda+\rho)$, and then find elements $w_i$ of the $i$-th factor of $W(f)$. It is easy to check that the $w_i$-s are determined by the scalar products $(\alpha, \lambda_f)$ for $\alpha\in\Delta(\lambda_f)$, and that these scalar products are in turn determined by the order $\triangleright$ introduced in Step 3 of the combinatorial algorithm. Thus the order $\triangleright$ encodes the elements $w_i$.

To compare Step 4 of the two algorithms, we have to ensure that Step 4 of the combinatorial algorithm implements correctly Step 4 of the mnemonic algorithm. This is accomplished by a careful reading of~\cite[Section: The Robinson-Schensted Algorithm for Classical Groups]{Barbasch-Vogan}.

It remains to compare Steps 5 of the two algorithms. The pair of partitions attached to the $W$-module $j_{W(f)}^W(E_{inf}(f))$ is the pair of partitions in the respective formula for the B/C/D-functions. This follows from~\cite{Lusztig}. The fact that the functions B/C/D compute correctly the partition attached to the orbit $\mathcal O(f)$ follows from the combinatorial description of the Springer correspondence for classical groups given in~\cite{Barbasch-Vogan}. Note that the functions B/C/D combine these two procedures in one formula.

\end{proof}

\subsection{Estimates on the corank of a partition}\label{SSSrrsa} Let $x\in\EuScript O(f)$. By identifying $\mathfrak g(n)$ with $\mathfrak g(n)^*$ we consider $x$ as a linear operator in a natural $\mathfrak g(n)$-module. Therefore we can define the {\it corank of} $x$ as the corank of the respective operator. For $x\in\EuScript O(f)$, the corank of $x$ is independent on $x$ and equals $\sharp p(f)$. 

\begin{lemma}\label{Lfcrk}Let $l$ be the length of a longest strictly decreasing subsequence of $f^+$. Then $$|\sharp p(f)-l|\le 5+1.$$\end{lemma}
\begin{proof}
It is known that, for each $i$, $\sharp p_i=\sharp p([\sim]_i)$ equals the length of a longest strictly decreasing subsequence of elements in $[\sim]_i$~\cite[p. 69, Ex. 7]{Knu}. A longest strictly decreasing subsequence of $f^+$ 
could be shorter by 1 than a longest strictly decreasing subsequence of elements of $[\sim]_i$ with respect to order $\triangleleft$: this is due to the exceptions in Step 3 for the $C/D$-cases. As a result, we have $$|\max(\sharp p_1, \sharp p_2, \sharp p_{\ge3})-l|=|\max(\sharp p_1, \sharp p_2, \sharp p_3,..., \sharp p_t)-l|\le1.$$
Combining this inequality with the inequalities B2/C2/D2 of~\ref{SSSin}, and recalling that $p(f)=(B/C/D)(p_1, p_2, p_{\ge3})$, we finish the proof of Lemma~\ref{Lfcrk}.
\end{proof}

\subsection{Proofs of Theorems~\ref{1Tbcd},~\ref{2T} and Proposition~\ref{Pp}}\label{Spr1tb}
These proofs are very similar to the proofs of corresponding statements in the A-case~\cite[Theorems 3.1, 3.2 and Proposition 3.3]{PP2}.
In particular, the proofs of Theorem~\ref{2T} and Proposition~\ref{Pp} coincide verbatim with the respective proofs of~\cite[Theorem~3.2]{PP2} and~\cite[Proposition~3.3]{PP2} modulo exchange of notation and replacing~\cite[Proposition 2.10]{PP2} by Proposition~\ref{Lexpl}.

We split the proof of Theorem~\ref{1Tbcd} into two parts:

a) if $\operatorname{Ann}_{\operatorname{U}(\mathfrak g(\infty))} \operatorname{L}_\mathfrak b(f) \ne 0$, then $f$ satisfies conditions (1) and (2) of Theorem~\ref{1Tbcd};

b) if $f$ satisfies conditions (1) and (2) of Theorem~\ref{1Tbcd}, then $$\operatorname{Ann}_{\operatorname{U}(\mathfrak g(\infty))} \operatorname{L}_\mathfrak b(f) \ne 0.$$\\
Furthermore, the proof of part a) can be broken down into the proofs of the following 3 statements:

$a1)$ Let $f\in\mathbb F^{\mathbb Z_{>0}}$. If $\operatorname{I}(f)\ne0$, then $|f|<\infty$.

$a2)$ Let $f\in\mathbb F^{\mathbb Z_{>0}}$ . If $\operatorname{I}(f)\ne0$, then $f$ is almost integral or almost half-integral.

$a3)$ Let $f\in\mathbb F^{\mathbb Z_{>0}}$ . If $\operatorname{I}(f)\ne0$, then $f$ is locally constant with
respect to linear order $\prec$.\\
Statement a1) coincides with Proposition~\ref{Papp} above, of which we provided a complete proof. The proofs of a2) and a3) follow very closely the respective proofs of Proposition~13 and~14 of~\cite{PP2}, where instead of~\cite[Lemma~18]{PP2} one has to use Lemma~\ref{Lfcrk}.

The proof of part b) is the proof of~\cite[Theorem 3.1 b)]{PP2} verbatim modulo the new $B/C/D$-notation, except in the case when $\mathfrak g(\infty)=\mathfrak{sp}(\infty)$ and $f$ is half-integral.

We now consider this latter case. Let $\mathfrak g(\infty)=\mathfrak g^C(\infty)=\mathfrak{sp}(\infty)$, and let $f_{\delta}$ to be a function such that $f_\delta(i)=\frac12$ for all $i\in\mathbb Z_{>0}$. One can check directly that $\operatorname{I}_\mathfrak b(f_\delta)=\operatorname{Ann}_{\operatorname{U}(\mathfrak{sp}(\infty))}\operatorname{L}_\mathfrak b(f_\delta)$

$\bullet$ does not depend on a choice of a splitting Borel subalgebra $\mathfrak b$,

$\bullet$ equals the kernel $I_W$ of the natural map $\operatorname{U}(\mathfrak{sp}(\infty))\to \mbox{Weyl}(\infty)$, where

Weyl$(\infty)$ is the Weyl algebra of $V(\infty)$ defined by the skew-symmetric form

of $V(\infty)$.\\
Since $f$ is half-integral, it is clear that $f-f_{\delta}$ is an almost integral function, and thus $\operatorname{L}_\mathfrak b(f-f_\delta)$ is annihilated by some proper integrable ideal $I$. Next, we observe that $\operatorname{L}_\mathfrak b(f)$ is a subquotient of $\operatorname{L}_\mathfrak b(f-f_\delta)\otimes \operatorname{L}_\mathfrak b(f_\delta)$ and thus
$$\operatorname{Ann}_{\operatorname{U}(\mathfrak g(\infty))} (\operatorname{L}_\mathfrak b(f-f_\delta)\otimes \operatorname{L}(f_\delta))\subset \operatorname{Ann}_{\operatorname{U}(\mathfrak g(\infty))} \operatorname{L}_\mathfrak b(f).$$In particular, if the left-hand side ideal is nonzero then the right-hand side ideal is also nonzero.

Now we prove that $$\operatorname{Ann}_{\operatorname{U}(\mathfrak g(\infty))} (\operatorname{L}_\mathfrak b(f-f_\delta)\otimes \operatorname{L}(f_\delta))\ne 0.$$
For this, we show  using Lemma~\ref{Ltd} 
that $\operatorname{Ann}_{\operatorname{U}(\mathfrak g(\infty))} (\operatorname{L}_\mathfrak b(f-f_\delta)\otimes \operatorname{L}(f_\delta))=D_\mathfrak o^\mathfrak{sp}I$ for some nonzero ideal $I$ of $\operatorname{U}(\mathfrak o(\infty))$. As $\operatorname{Ann}_{\operatorname{U}(\mathfrak g(\infty))} (\operatorname{L}_\mathfrak b(f-f_\delta)$) is an integrable ideal,
$$\operatorname{Ann}_{\operatorname{U}(\mathfrak g(\infty))} (\operatorname{L}_\mathfrak b(f-f_\delta))=D_\mathfrak o^\mathfrak{sp}I'$$for some integrable ideal $I'$ of $\operatorname{U}(\mathfrak o(\infty))$. One the other hand,$$\operatorname{Ann}_{\operatorname{U}(\mathfrak g(\infty))} (\operatorname{L}_\mathfrak b(f_\delta)=D_\mathfrak o^{\mathfrak{sp}}\operatorname{Ann}_{\operatorname{U}(\mathfrak o(\infty))} \mbox{SW}.$$ Hence, indeed $$\operatorname{Ann}_{\operatorname{U}(\mathfrak g(\infty))} (\operatorname{L}_\mathfrak b(f-f_\delta)\otimes \operatorname{L}(f_\delta))=D_\mathfrak o^{\mathfrak{sp}}I$$for some nonzero integrable ideal $I$ in $\operatorname{U}(\mathfrak o(\infty))$.

\section{Integrable and semiintegrable ideals are radical}
One can define the {\it radical} $\sqrt I$ of an ideal $I$ by one of the following requirements:

1. $\sqrt I$ is the intersection of all primitive ideals which contain $I$,

2. $\sqrt I$ is the intersection of all prime ideals which contain $I$,

3. $\sqrt I$ is the sum of all ideals $J$ such that $J^n\subset I$ for some $n$.
\begin{proposition}\label{P2}If $\mathfrak g(\infty)=\mathfrak{sl}(\infty), \mathfrak o(\infty), \mathfrak{sp}(\infty)$ and $I\subset\operatorname{U}(\mathfrak g(\infty))$ is an integrable ideal, then definitions 1, 2, 3 are equivalent, and moreover $I=\sqrt I$.\end{proposition}
\begin{proof}Any integrable ideal is an intersection of finitely many prime integrable ideals, and any prime integrable ideal is primitive, see Proposition~\ref{P44} b). This shows that $I=\sqrt I$ with respect to definitions 1 and 2.

To prove that $I=\sqrt I$ with respect to definition 3, it is enough to show that $I\cap\operatorname{U}(\mathfrak g')=\sqrt{I\cap\operatorname{U}(\mathfrak g')}$ for any finite-dimensional subalgebra $\mathfrak g'\subset\mathfrak g(\infty)$. The last statement follows from the fact that the ideal $I\cap\operatorname{U}(\mathfrak g')$ is an intersection of prime ideals as it is integrable.\end{proof}
For $\mathfrak g(\infty)=\mathfrak{sp}(\infty)$ we have a slightly more general statement. We define an ideal $I\subset\operatorname{U}(\mathfrak{sp}(\infty))$ to be {\it semiintegrable} if $I$ is in the image of the lattice of integrable ideals in $\operatorname{U}(\mathfrak o(\infty))$ under the isomorphism of lattices constructed in the proof of Theorem~\ref{Tilosp}. 
A semiintegrable ideal may be integrable.

\begin{proposition}\label{P2X}If $\mathfrak g(\infty)=\mathfrak{sp}(\infty)$, $I\subset\operatorname{U}(\mathfrak g(\infty))$ is a semiintegrable ideal, and $\mathbb F$ is uncountable, then definitions 1, 2, 3 are equivalent, and moreover $I=\sqrt I$.\end{proposition}
\begin{proof}\footnote{After this paper was completed, we found a similar argument in~\cite[Chapter~9]{MCR}, so we present the proof here for the mere convenience of the reader.} Consider definition 3 first. Clearly, it suffices to show that if $\sqrt I$ is defined as in definition 3, then $I\cap\operatorname{U}(\mathfrak g(S))=\sqrt{I\cap\operatorname{U}(\mathfrak g(S))}$ for any finite subset $S\subset\mathbb Z_{>0}$.

Recall that $\mathfrak{sp}(2n)$ has two nonisomorphic Shale-Weil (oscillator) representations with respective highest weights $\frac12\varepsilon_1+...+\frac12\varepsilon_n$ and $\frac12\varepsilon_1+...+\frac12\varepsilon_{n-1}-\frac12\varepsilon_n$. It is easy to check that $\mathfrak{sp}(\infty)$ has infinitely many nonisomorphic Shale-Weil representations obtained as direct limits of Shale-Weil representations of $\mathfrak{sp}(2n)$. Since the annihilators of all these $\mathfrak{sp}(\infty)$-modules coincide, for our purposes  it suffices to consider one fixed Shale-Weil $\mathfrak{sp}(\infty)$-module which we denote by SW.

Proposition~\ref{Ltd} implies that a semiintegrable ideal is the annihilator of an $\mathfrak{sp}(\infty)$-module of the form $M\oplus N\otimes\mbox{SW}$, where $M$ are $N$ integrable $\mathfrak{sp}(\infty)$-modules. Therefore, for a finite set $S$, we have
$$I\cap\mathfrak g(S)=\operatorname{Ann}_{\operatorname{U}(\mathfrak g(S))}(M|_{\mathfrak g(S)}\oplus N|_{\mathfrak g(S)}\otimes \mbox{SW}|_{\mathfrak g(S)}).$$
Note that the restriction of $\mbox{SW}$ to $\mathfrak g(S)$ is a direct sum of infinitely many copies of Shale-Weil representations of $\mathfrak g(S)$. A Shale-Weil representation of $\mathfrak g(S)$ is a (highest) weight module with 1-dimensional weight spaces.
Thus a tensor product of a Shale-Weil representation of $\mathfrak g(S)$ with a finite-dimensional $\mathfrak g(S)$-module is a direct sum of finitely many bounded weight modules of finite length, each of which affords a generalized central character. The annihilator of any such a simple module $P$ equals the annihilator of a simple module in the block of $P$, see~\cite[Theorems 5.1, 5.2]{GS}. Therefore $I\cap\mathfrak g(S)$ is the intersection of some set of primitive ideals, and hence \begin{equation}I\cap\operatorname{U}(\mathfrak g(S))=\sqrt{I\cap\operatorname{U}(\mathfrak g(S))}.\label{Erad}\end{equation}

Next we show that $I=\sqrt I$ with respect to definition 1. This will automatically imply that $I=\sqrt I$ with respect to definition 2. Let $\sqrt I$ be the radical of $I$ with respect to definition 1. We prove first (following closely~\cite[3.1.15]{Dix}) that for any $i\in \sqrt I$ there exists $n\in\mathbb Z_{>0}$ such that $i^n\in I$.

Fix $i\in\sqrt I/I$ and consider the algebra $C:=(\operatorname{U}(\mathfrak g)/I)\otimes \mathbb F[X]$, where $X$ is a new variable. Then $C(1-iX)=C$ or $C(1-iX)\ne C$, where $C(1-iX)$ denotes the left ideal of $C$ generated by $(1-iX)$. If $C(1-iX)=C$, there exist $a_0, a_1,..., a_n$ such that $$(1-iX)(a_0+a_1X+a_2X^2+...+a_nX^n)=1.$$ Consequently, $a_s=i^s$ for $s\le n$ and $i^{n+1}=0$, which is precisely what we need to prove.

Assume next that $C(1-iX)\ne C$. Then there is a simple $C$-module $M$ and an element $m\in M$ such that $(1-iX)m=0$. We claim that there exists $\lambda\in \mathbb F$ such that $Xm'=\lambda m'$ for any $m'\in M$, or equivalently  such that $(X-\lambda)M=0$. Indeed, assume to the contrary that, for any $\lambda\in\mathbb F$, the homomorphism $$\phi_\lambda: M\to M,\hspace{10pt}m'\mapsto (X-\lambda)m',$$ is nonzero. The fact that $X-\lambda$ belongs to the center of $C$ implies that the kernel and cokernel
of $\phi_\lambda$ equal 0, and therefore that $\phi_\lambda$ is an automorphism of $M$ for any $\lambda$. The collection of elements
$$\{\phi^{-1}_\lambda m\}_{\lambda\in\mathbb F}$$ is uncountable as $\mathbb F$ is uncountable. On the other hand, $M$ is at most countable dimensional over $\mathbb F$ because $C$ is countable dimensional over $\mathbb F$. Thus there exist nonzero sequences $\lambda_1, \lambda_2, ..., \lambda_n\in\mathbb F$ and $\alpha_1, ..., \alpha_n\in\mathbb F$ such that
$$\Sigma_i\alpha_i\phi_{\lambda_i}^{-1}m=0.$$
Clearly, $\Sigma_i\alpha_i\phi_{\lambda_i}^{-1}$ is an endomorphism of $M$, and hence $\Sigma_i\alpha_i\phi_{\lambda_i}^{-1}M=0$. Therefore $P(X)M=0$, where \begin{center}$P(X)=(X-\lambda_1)(X-\lambda_2)...(X-\lambda_n){\Sigma}_i\frac{\alpha_i}{X-\lambda_i}.$\end{center} This implies that $(X-\lambda)M=0$ for some root $\lambda\in\mathbb F$ of the polynomial $P(X)$.

Finally, we have $$0=(1-iX)m=m-\lambda im,$$and thus $im\ne 0$. Consequently, $i$ does not annihilate $M$, and hence $i\not\in\sqrt I/I$. This shows that our assumption is contradictory, and as a consequence we obtain that for any $i\in\sqrt I$ there exists $n$ such that $i^n\in I$.

Next, one shows exactly as in~\cite[3.1.15]{Dix} that for any finite set $S$ there exists $n\in\mathbb Z_{>0}$ such that
$$(\sqrt I\cap\operatorname{U}(\mathfrak g(S)))^n\subset I\cap\operatorname{U}(\mathfrak g(S)).$$ This, together with~(\ref{Erad}), implies

$$(\sqrt I\cap\operatorname{U}(\mathfrak g(S)))^n\supset \sqrt{I\cap\operatorname{U}(\mathfrak g(S))}=I\cap\operatorname{U}(\mathfrak g(S)).$$ Therefore, $\sqrt I\cap\operatorname{U}(\mathfrak g(S))=I\cap\operatorname{U}(\mathfrak g(S))$ for any finite set $S$, and hence $\sqrt I=I$.

\end{proof}

\section*{Appendix A: Roots, weights, and splitting Borel subalgebras}
\label{SSsodef}
The Lie algebra $\mathfrak{gl}(\infty)$ can be defined  as the Lie algebra of infinite matrices $(a_{ij})_{i, j\in\mathbb Z}$ each of which has at most finitely many nonzero entries. Equivalently, $\mathfrak{gl}(\infty)$ can be defined by giving an explicit basis. Let $\{e_{ij}\}_{i, j\in\mathbb Z}$ be a basis of a countable-dimensional vector space denoted by $\mathfrak{gl}(\infty)$. The structure of a Lie algebra on $\mathfrak{gl}(\infty)$ is given by the formula
$$[e_{ij},
e_{kl}]=\delta_{jk}e_{il}-\delta_{il}e_{kj},$$
where $i, j, k, l\in\mathbb Z$ and $\delta_{mn}$ is Kronecker's delta.

The Lie algebras $\mathfrak o^{B/D}(\infty)$ and $\mathfrak{sp}(\infty)$ can be defined as subalgebras of $\mathfrak{gl}(\infty)$ spanned by the following vectors

$$\begin{array}{c|cccc}
\mathfrak o^B(\infty)&e^B_{i, -j}:=e_{i, j}-e_{-j, -i},&\begin{array}{c}e^B_{i, j}:=e_{i, -j}-e_{j, -i},\\e^B_{-i, -j}:=e_{-i, j}-e_{-j, i},\end{array}&e^B_{\pm i}:=e_{\pm i, 0}-e_{0, \mp i}\\\hline
\mathfrak o^D(\infty)&e^D_{i, -j}:=e_{i, j}-e_{-j, -i},&e^D_{i, j}:=e_{i, -j}-e_{j, -i},&e^D_{-i, -j}:=e_{-i, j}-e_{-j, i}&\\\hline
\mathfrak{sp}(\infty)&e^C_{i, -j}:=e_{i, j}-e_{-j, -i},&e^C_{i, j}:=e_{i, -j}+e_{j, -i},&e^C_{-i, -j}:=e_{-i, j}+e_{-j, i},&\\
\end{array}$$for $i, j\in\mathbb Z_{>0}$.
We set $$\mathfrak g^B(\infty):=\mathfrak o^B(\infty),\hspace{10pt}\mathfrak g^C(\infty):=\mathfrak{sp}(\infty),\hspace{10pt} \mathfrak g^D(\infty):=\mathfrak o^D(\infty).$$
Note that the Lie algebra spanned by $\{e_{i, -j}^B\}_{i, j\in\mathbb Z_{>0}}$  is isomorphic to $\mathfrak{gl}(\infty)$, and let $\mathfrak g^A(\infty)\cong\mathfrak{sl}(\infty)$ be the commutator subalgebra of this Lie algebra.

The splitting Cartan subalgebras introduced in Subsection~\ref{SSspl} can be chosen as follows:
\begin{center}$\mathfrak h^A:=\PenkovPetukhovBr{e_{i, -i}^B-e_{j, -j}^B}_{i, j\in\mathbb Z_{>0}}$, $\mathfrak h^{B/C/D}:=\PenkovPetukhovBr{e^{B}_{i, -i}}_{i\in\mathbb Z_{>0}}$.\end{center}
Then the Lie algebra $\mathfrak g^{A/B/C/D}(\infty)$ has the root decomposition $$\mathfrak g^{A/B/C/D}(\infty)=\mathfrak h^{A/B/C/D}\oplus\bigoplus_{\alpha\in\Delta^{B/C/D}}\mathfrak g^{A/B/C/D}(\infty)^{\alpha}$$
which is similar to the usual root decomposition  respectively of $\mathfrak{sl}(n)$, $\mathfrak o(2n+1)$, $\mathfrak{sp}(2n)$ and $\mathfrak o(2n)$. Here
$$\Delta^A:=\{\varepsilon_i-\varepsilon_j\}_{i\ne j\in\mathbb Z_{>0}},\hspace{74pt}\Delta^B:=\{\varepsilon_i-\varepsilon_j, \pm\varepsilon_i, \pm(\varepsilon_i+\varepsilon_j)_{i\ne j}\}_{i, j\in\mathbb Z_{>0}},$$
$$\Delta^C:=\{\varepsilon_i-\varepsilon_j, \pm2\varepsilon_i, \pm(\varepsilon_i+\varepsilon_j)_{i\ne j}\}_{i, j\in\mathbb Z_{>0}},\hspace{10pt}\Delta^D:=\{\varepsilon_i-\varepsilon_j, \pm(\varepsilon_i+\varepsilon_j)_{i\ne j}\}_{i, j\in\mathbb Z_{>0}},$$
where the system of vectors $\{\varepsilon_j\}_{j\in\mathbb Z_{>0}}$ in $(\mathfrak h^{B/C/D})^*$ is dual to the basis $\{e_{i, -i}\}_{i\in\mathbb Z_{>0}}$ of $\mathfrak h^{B/C/D}$, and the system of vectors $\{\varepsilon_j\}_{j\in\mathbb Z_{>0}}$ for $\mathfrak h^A$ is the restriction of $\{\varepsilon_j\}_{j\in\mathbb Z_{>0}}$ from $\mathfrak h^{B/C/D}$ to $\mathfrak h^A$.

A {\it splitting Borel subalgebra} $\mathfrak b\subset\mathfrak g^{A/B/C/D}(\infty)$ is defined as the inductive limit of Borel subalgebras $\mathfrak b(n)\subset\mathfrak g(n)$ in the sequence~(\ref{Ech1}). Any splitting Borel subalgebra is conjugate via $\operatorname{Aut}(\mathfrak g(\infty))$ to a splitting Borel subalgebra containing $\mathfrak h^{A/B/C/D}$, and we only consider splitting Borel subalgebras $\mathfrak b$ satisfying this assumption. Fixing $\mathfrak b$ is equivalent to splitting $\Delta=\Delta^{A/B/C/D}$ into $\Delta^+\sqcup\Delta^-$ with the usual properties

$\bullet$ $\alpha, \beta\in \Delta^{\pm}, \alpha+\beta\in\Delta \Rightarrow \alpha+\beta\in\Delta^{\pm}$,

$\bullet$ $\alpha\in\Delta^{\pm}\Leftrightarrow-\alpha\in\Delta^{\mp}$.

It has been observed in~\cite{DP} that the splitting Borel subalgebras of $\mathfrak{sl}(\infty)$ containing $\mathfrak h^A$ are in one-to-one correspondence with linear orders $\prec$ on the set $\mathbb Z_{>0}$: given such an order, the corresponding set of positive roots is
$$\Delta^{A+}(\prec):=\{\varepsilon_i-\varepsilon_j\}_{i\prec j}.$$
In a similar way, the splitting Borel subalgebras of $\mathfrak o^B(\infty)$ and $\mathfrak{sp}(\infty)$ containing $\mathfrak h^{B/C}$ are in one-to-one correspondence with linear orders on $\mathbb Z_{>0}$ together with a partition $S_+\sqcup S_-$ of $\mathbb Z_+$: given such datum, the corresponding subset of positive roots is
\begin{center}$\{\varepsilon_i+\varepsilon_j\}_{i, j\in S_+}\cup\{\varepsilon_i-\varepsilon_j\}_{i\in S_+, j\in S_-}\cup\{-\varepsilon_i-\varepsilon_j\}_{i, j\in S_-}\cup\{\varepsilon_i+\varepsilon_j\}_{i\in S_+\prec j\in S_-}\cup$\\$\cup\{\varepsilon_i-\varepsilon_j\}_{i\in S_+\prec j\in S_+}\cup\{\varepsilon_i-\varepsilon_j\}_{i\in S_-\succ j\in S_-} \cup\{\varepsilon_i\}_{i\in S_+}\cup\{-\varepsilon_i\}_{i\in S_-}$\end{center} in the $B$-case, and\begin{center}$\{\varepsilon_i+\varepsilon_j\}_{i, j\in S_+}\cup\{\varepsilon_i-\varepsilon_j\}_{i\in S_+, j\in S_-}\cup\{-\varepsilon_i-\varepsilon_j\}_{i, j\in S_-}\cup\{\varepsilon_i+\varepsilon_j\}_{i\in S_+\prec j\in S_-}\cup$\\$\cup\{\varepsilon_i-\varepsilon_j\}_{i\in S_+\prec j\in S_+}\cup\{\varepsilon_i-\varepsilon_j\}_{i\in S_-\succ j\in S_-} \cup\{2\varepsilon_i\}_{i\in S_+}\cup\{-2\varepsilon_i\}_{i\in S_-}$\end{center} in the $C$-case.

The splitting Borel subalgebras of $\mathfrak o^D(\infty)$ containing $\mathfrak h$ are in one-to-one correspondence with linear orders on $\mathbb Z_{>0}$ together with a partition $S_+\sqcup S_-\sqcup S_0$ of $\mathbb Z_+$ such that $S_0$ is the set of $\prec$-maximal elements (thus $S_0$ consists of one element or is empty): given such datum, the corresponding subset of positive roots is \begin{center}$
\{\varepsilon_i+\varepsilon_j\}_{i\in S_+\sqcup S_0, j\in S_+}\cup\{\varepsilon_i-\varepsilon_j\}_{i\in S_+\sqcup S_0, j\in S_-\sqcup S_0}\cup\{-\varepsilon_i-\varepsilon_j\}_{i\in S_-\sqcup S_0, j\in S_-}\cup$\\$\cup\{\varepsilon_i+\varepsilon_j\}_{i\in S_+\prec j\in S_-}\cup \{\varepsilon_i-\varepsilon_j\}_{i\in S_+\prec j\in S_+\sqcup S_0}\cup\{\varepsilon_i-\varepsilon_j\}_{i\in S_-\succ j\in S_-}$.\end{center}
(In the paper~\cite[p.~229]{DP} an equivalent description of splitting Borel subalgebras of $\mathfrak g^{B/C/D}$ is given. It is based on the notion of $\mathbb Z_2$-linear order which we do not use here).

It is easy to verify that, for any splitting Borel subalgebra $\mathfrak b$, there is an automorphism $w\in \operatorname{Aut}(\mathfrak g^{B/C/D}(\infty))$ such that $w\mathfrak h^{B/C/D}=\mathfrak h^{B/C/D}$ and $S_-=\emptyset$ for $w\mathfrak b$. Hence for the purposes of this paper it suffices to consider only the case in which $S_-=\emptyset$. Under this assumption, a linear order $\prec$ on $\mathbb Z_{>0}$ determines a unique Borel subalgebra:
$$\begin{tabular}{cc}$\Delta^{B+}(\prec):=$&$\{\varepsilon_i\}_{i\in\mathbb Z_{>0}}\sqcup \{\varepsilon_i+\varepsilon_j\}_{i\ne j\in \mathbb Z_{>0}}\sqcup \{\varepsilon_i-\varepsilon_j\}_{i\in \mathbb Z_{>0}\prec j\in\mathbb Z_{>0}},$\\$\Delta^{C+}(\prec):=$&$\{2\varepsilon_i\}_{i\in\mathbb Z_{>0}}\sqcup\{\varepsilon_i+\varepsilon_j\}_{i\ne j\in \mathbb Z_{>0}}\sqcup \{\varepsilon_i-\varepsilon_j\}_{i\in \mathbb Z_{>0}\prec j\in\mathbb Z_{>0}},$\\$\Delta^{D+}(\prec):=$&$\{\varepsilon_i+\varepsilon_j\}_{i\ne j\in \mathbb Z_{>0}}\sqcup \{\varepsilon_i-\varepsilon_j\}_{i\in \mathbb Z_{>0}\prec j\in \mathbb Z_{>0}}.$\end{tabular}$$

Finally, we need to give a definition of dominant function as used in Lemma~\ref{L56}. A function $f\in\mathbb F^{\mathbb Z_{>0}}$ is {\it $\mathfrak b$-dominant} for fixed $\mathfrak b(\prec)$, if
$$\begin{tabular}{ll} (A) $f(i)-f(j)\in\mathbb Z_{\ge0}$ for $i\prec j\in\mathbb Z_{>0}$ & case $A$;\\
(B1) $f(i)\in\mathbb Z_{\ge0}$ for all $i\in\mathbb Z_{>0}$ or $f(i)\in\frac12+\mathbb Z_{\ge0}$ for all $i\in\mathbb Z_{>0}$,& case $B$;\\
(B2) $f(i)\ge f(j)$ for $i\prec j$&\\
(C1) $f(i)\in\mathbb Z_{\ge0}$ for all $i\in \mathbb Z_{>0}$,&case $C$;\\
(C2) $f(i)\ge f(j)$ for $i\prec j$&\\
(D1) $\begin{tabular}{c}$f(i)\in\mathbb Z$ for all $i\in\mathbb Z_{>0}$ or $f(i)\in\frac12+\mathbb Z$ for all $i\in\mathbb Z_{>0}$,~and\\ $f(i)\ge0$ for all $i\in \mathbb Z_{>0}$ which are not $\prec$-maximal,\end{tabular}$&case $D$.\\
(D2) $|f(i)|\ge |f(j)|$ for $i\prec j$&\end{tabular}$$



\section*{Appendix B: T-algebras and $\mathfrak{osp}$-duality}
Let $\mathcal C$ be a tensor (or monoidal) category. We define a {\it T-algebra in}  $\mathcal C$ to be an object $M$ of $\mathcal C$ together with a morphism $m: M\otimes M\to M$. Two T-algebras $M_1$ and $M_2$, in respective tensor categories $\mathcal C_1, \mathcal C_2$, are {\it isomorphic} if there exists an equivalence of tensor categories $\epsilon: \mathcal C_1 \to\mathcal C_2$ such that $\epsilon(M_1)=M_2$ and $\epsilon(m_1)=m_2$.

For example, an algebra over a field $\mathbb F$ (or a commutative ring $R$) is a T-algebra in the tensor category of $\mathbb F$-vector spaces (respectively, $R$-modules).

If $\mathfrak g$ is a Lie algebra, then the enveloping algebra $\operatorname{U}(\mathfrak g)$ defines a T-algebra $\operatorname{TU}(\mathfrak g)$ in the category of $\mathfrak g$-modules via the morphism $\operatorname{U}(\mathfrak g)\otimes \operatorname{U}(\mathfrak g)\to\operatorname{U}(\mathfrak g)$.

A left (respectively, right, or two-sided) ideal in a T-algebra $M$ is a subobject $I$ of $M$ such that $m$ maps $M\otimes I$ to $I$ (respectively, $I\otimes M$ to $I$, or both $M\otimes I$ and $I\otimes M$ to $I$). The following lemma is straightforward.

\begin{lemma}\label{L73}a) The notions of left, right and two-sided ideals in $\operatorname{TU}(\mathfrak g)$ coincide.\\
b) The lattice of (two-sided) ideals in $\operatorname{U}(\mathfrak g)$ is naturally isomorphic to the lattice of ideals in $\operatorname{TU}(\mathfrak g)$.
\end{lemma}
Recall that $\operatorname{U}(\mathfrak g)$ is the quotient of the tensor algebra T$^\cdot(\mathfrak g)$ by the ideal generated by $x\otimes y-y\otimes x-[x, y]$. It is clear that $\operatorname{TU}(\mathfrak g)$ is the quotient of the $T$-algebra TT$^\cdot(\mathfrak g)$  by the ideal generated by the image of the morphism
$$\psi: \mathfrak g\otimes \mathfrak g\to (\mathfrak g\otimes\mathfrak g)\oplus\mathfrak g\subset \operatorname{TT}^\cdot(\mathfrak g)\hspace{10pt}(x\otimes y\to x\otimes y-y\otimes x-[x, y]).$$
In what follows we refer to this image as the {\it module of $\mathfrak g$-relations}. Note that the morphism $\psi$ factors through the morphism $\mathfrak g\otimes\mathfrak g\to{\bf \Lambda}^2\mathfrak g$.

Let $\mathbb T_{\mathfrak o(\infty)}^{ind}$ be the category of inductive limits of objects from $\mathbb T_{\mathfrak o(\infty)}$. Similarly, we define the category $\mathbb T_{\mathfrak{sp}(\infty)}^{ind}$. Note that the T-algebras $\operatorname{TU}(\mathfrak o(\infty))$ and $\operatorname{TU}(\mathfrak{sp}(\infty))$ are well defined in the respective categories $\mathbb T_{\mathfrak o(\infty)}^{ind}$ and $\mathbb T_{\mathfrak{sp}(\infty)}^{ind}$.
\begin{theorem}\label{Tosp}The T-algebras $\operatorname{TU}(\mathfrak o(\infty))$ and $\operatorname{TU}(\mathfrak{sp}(\infty))$ are isomorphic.\end{theorem}



In~\cite{S} V.~Serganova has constructed an explicit functor $D_\mathfrak o^{\mathfrak{sp}}: \mathbb T_{\mathfrak o(\infty)}\to\mathbb T_{\mathfrak{sp}(\infty)}$ which is an equivalence of tensor categories. It is clear that this functor induces also an equivalence of the tensor categories $\mathbb T_{\mathfrak o(\infty)}^{ind}$ and $\mathbb T_{\mathfrak{sp}(\infty)}^{ind}$.
In order to prove Theorem~\ref{Tosp}, it is enough to show that $D_{\mathfrak o}^{\mathfrak{sp}}\operatorname{TU}(\mathfrak o(\infty))=\operatorname{TU}(\mathfrak{sp}(\infty))$ and that $D_\mathfrak o^{\mathfrak{sp}}m_\mathfrak o=m_{\mathfrak{sp}}$, where $m_\mathfrak o$ is the multiplication morphism for $\operatorname{TU}(\mathfrak o(\infty))$ and $m_{\mathfrak{sp}}$ is the multiplication morphism for $\operatorname{TU}(\mathfrak{sp}(\infty))$.
However, $D_\mathfrak o^\mathfrak{sp}$T$^\cdot(\mathfrak o(\infty))=$T$^\cdot(\mathfrak{sp}(\infty))$, and since $D_\mathfrak o^\mathfrak{sp}$ is a tensor functor, it suffices to show that $D_{\mathfrak o}^{\mathfrak{sp}}$ maps the module of $\mathfrak o(\infty)$-relations in TT$^\cdot(\mathfrak o(\infty))$ to the module of $\mathfrak{sp}(\infty)$-relations in TT$^\cdot(\mathfrak{sp}(\infty))$.

We need the following two lemmas.
\begin{lemma}\label{L75}We have $D_{\mathfrak o}^{\mathfrak{sp}}{\bf \Lambda}^2({\bf \Lambda}^2V(\infty))={\bf \Lambda}^2(\operatorname{\bf S}^2V(\infty))$.\end{lemma}
\begin{proof} The idea is to embed ${\bf \Lambda}^2({\bf \Lambda}^2V(\infty))$ into $$V(\infty)^{\otimes 4}:=V(\infty)\otimes V(\infty)\otimes V(\infty)\otimes V(\infty)$$ and then show that $D_{\mathfrak o}^{\mathfrak{sp}}$ maps ${\bf \Lambda}^2({\bf \Lambda}^2V(\infty))$ to ${\bf \Lambda}^2(\operatorname{\bf S}^2V(\infty))$ as submodules of $V(\infty)^{\otimes 4}$.

Since $D_\mathfrak o^{\mathfrak{sp}}V(\infty)=V(\infty)$, we see that $D_\mathfrak o^{\mathfrak{sp}}$ maps $V(\infty)^{\otimes 4}$ to $V(\infty)^{\otimes 4}$. Next, it is easy to check that, for any permutation $\sigma$ of the set $\{1, 2, 3, 4\}$,
$$D_{\mathfrak o}^{\mathfrak{sp}}\sigma=\mbox{sgn}(\sigma)\sigma$$ where $\sigma$ is considered as a linear operator on $V(\infty)^{\otimes 4}$. In what follows we write $\sigma_{\mathfrak o}$ and $\sigma_{\mathfrak{sp}}$ to distinguish the action of $\sigma$ on the fourth tensor powers of the natural representations of $\mathfrak{o}(\infty)$ and $\mathfrak{sp}(\infty)$.

The $\mathfrak o(\infty)$-module ${\bf \Lambda}^2({\bf \Lambda}^2V(\infty))$ is nothing but the $\mathfrak o(\infty)$-submodule of $V(\infty)^{\otimes 4}$ consisting of tensors $R$ such that
$$(12)_{\mathfrak o}R=-R,\hspace{10pt}(34)_{\mathfrak o}R=-R,\hspace{10pt}((13)(24))_{\mathfrak o}R=-R.$$
Consequently, $D_\mathfrak o^{\mathfrak{sp}}{\bf \Lambda}^2({\bf \Lambda}^2V(\infty))$ is the $\mathfrak{sp}(\infty)$-submodule of $V(\infty)^{\otimes 4}$ consisting of tensors $X$ such that
$$(12)_{\mathfrak{sp}}X=X,\hspace{10pt}(34)_{\mathfrak{sp}}X=X,\hspace{10pt}((13)(24))_{\mathfrak{sp}}X=-X.$$
This latter $\mathfrak{sp}(\infty)$-submodule is nothing but ${\bf \Lambda}^2(\operatorname{\bf S}^2V(\infty))$, and the proof is complete.\end{proof}
\begin{lemma}\label{Lhom}We have $\dim\operatorname{Hom}_{\mathfrak{sp}(\infty)}({\bf \Lambda}^2(\operatorname{\bf S}^2V(\infty)), \operatorname{\bf S}^2V(\infty))=1$.\end{lemma}
\begin{proof} The $\mathfrak{sp}(\infty)$-module ${\bf \Lambda}^2(\operatorname{\bf S}^2V(\infty))$ is a direct summand of $V(\infty)^{\otimes 4}$. The emebedding is given by the formula
$$(xy)\wedge (zt)\mapsto [(x\otimes y+y\otimes x)\otimes(z\otimes t+t\otimes z)-(z\otimes t+t\otimes z)\otimes(x\otimes y+y\otimes x)].$$
Restriction from $V(\infty)^{\otimes 4}$ to ${\bf \Lambda}^2(\operatorname{\bf S}^2V(\infty))$ defines a surjective linear operator $$\operatorname{Hom}_{\mathfrak{sp}(\infty)}(V(\infty)^{\otimes 4}, V(\infty)^{\otimes 2})\to \operatorname{Hom}_{\mathfrak{sp}(\infty)}({\bf \Lambda}^2(\operatorname{S}^2V(\infty)), V(\infty)^{\otimes 2}).$$
A basis of the space $\operatorname{Hom}_{\mathfrak{sp}(\infty)}(V(\infty)^{\otimes 4}, V(\infty)^{\otimes 2})$ is given in~\cite[Lemma~6.1]{PSt}, and it is straightforward to check that all basis elements map to a single 1-dimensional subspace of $\operatorname{Hom}_{\mathfrak{sp}(\infty)}({\bf \Lambda}^2(\operatorname{S}^2V(\infty)), V(\infty)^{\otimes 2})$. Hence, $$\dim\operatorname{Hom}_{\mathfrak{sp}(\infty)}({\bf \Lambda}^2(\operatorname{S}^2V(\infty)), \operatorname{S}^2V(\infty))\le1.$$

On the other hand, $\operatorname{S}^2V(\infty)\cong\mathfrak{sp}(\infty)$, and the existence of the Lie bracket on $\operatorname{S}^2V(\infty)$ shows that $$\dim\operatorname{Hom}_{\mathfrak{sp}(\infty)}({\bf \Lambda}^2(\operatorname{S}^2V(\infty)), \operatorname{S}^2V(\infty))\ge1.$$ This completes the proof.\end{proof}
\begin{proof}[Proof of Theorem~\ref{Tosp}] Lemma~\ref{L75} implies that the module of $\mathfrak o(\infty)$-relations is being mapped by $D_\mathfrak o^{\mathfrak{sp}}$ to the image of a homomorphism of the form

$${\bf \Lambda}^2\mathfrak{sp}(\infty)\to [{\bf \Lambda}^2\mathfrak{sp}(\infty)\oplus\mathfrak{sp}(\infty)]\subset[\mathfrak{sp}(\infty)^{\otimes2}\oplus\mathfrak{sp}(\infty)]\subset \operatorname{T}^{\cdot}
(\mathfrak{sp}(\infty)),\hspace{10pt} x\mapsto x-\phi(x),$$
where $\phi:{\bf \Lambda}^2\mathfrak{sp}(\infty)\to\mathfrak{sp}(\infty)$ is the image of the Lie bracket morphism $${\bf \Lambda}^2\mathfrak o(\infty)\to\mathfrak o(\infty)$$ under $D_\mathfrak o^\mathfrak{sp}$. Lemma~\ref{Lhom} claims that, up to proportionality, $\phi$ coincides with the Lie bracket morphism for $\mathfrak{sp}(\infty)$. This completes the proof.
\end{proof}
Theorem~\ref{Tilosp} is a direct consequence of Theorem~\ref{Tosp} and Lemma~\ref{L73}. Note that, despite the fact that $\operatorname{TU}(\mathfrak o(\infty))$ and $\operatorname{TU}(\mathfrak{sp}(\infty))$ are isomorphic T-algebras, the algebras $\operatorname{U}(\mathfrak{o}(\infty))$ and $\operatorname{U}(\mathfrak{sp}(\infty))$ are not isomorphic, see~\cite{PP1}.

The following proposition is used in Proposition~\ref{P2X} above.
\begin{proposition}\label{Ltd}Let $X_\mathfrak o$ and $Y_\mathfrak o$ be $\mathfrak o(\infty)$-modules, $X_{\mathfrak{sp}}$ and $Y_{\mathfrak{sp}}$ be $\mathfrak{sp}(\infty)$-modules, such that $$D_\mathfrak o^\mathfrak{sp}(\operatorname{Ann}_{\operatorname{U}(\mathfrak o(\infty))}X_\mathfrak o)=\operatorname{Ann}_{\operatorname{U}(\mathfrak{sp}(\infty))}X_\mathfrak{sp},\hspace{10pt}D_\mathfrak o^\mathfrak{sp}(\operatorname{Ann}_{\operatorname{U}(\mathfrak o(\infty))}Y_\mathfrak o)=\operatorname{Ann}_{\operatorname{U}(\mathfrak{sp}(\infty))}Y_\mathfrak{sp}.$$ Then $D_\mathfrak o^\mathfrak{sp}(\operatorname{Ann}_{\operatorname{U}(\mathfrak o(\infty))}(X_\mathfrak o\otimes Y_\mathfrak o))=\operatorname{Ann}_{\operatorname{U}(\mathfrak{sp}(\infty))}(X_\mathfrak{sp}\otimes Y_{\mathfrak{sp}})$.\end{proposition}
\begin{proof}
Let $\Delta_\mathfrak o$ be the diagonal morphism of Lie algebras $\mathfrak o(\infty)\to\mathfrak o(\infty)\oplus\mathfrak o(\infty)$. This morphism induces the comultiplication morphism
$$\Delta_\mathfrak o^U: \operatorname{U}(\mathfrak o(\infty))\to \operatorname{U}(\mathfrak o(\infty))\otimes\operatorname{U}(\mathfrak o(\infty)).$$We have
\begin{eqnarray}\begin{aligned}\label{Fl}\operatorname{Ann}_{\operatorname{U}(\mathfrak o(\infty))}(X_\mathfrak o\otimes Y_\mathfrak o)=(\Delta^U_\mathfrak o)^{-1}(\operatorname{Ann}_{\operatorname{U}(\mathfrak o(\infty))}X_\mathfrak o\otimes\operatorname{U}(\mathfrak o(\infty)_r)+\\+\operatorname{U}(\mathfrak o(\infty)_l)\otimes\operatorname{Ann}_{\operatorname{U}(\mathfrak o(\infty))}Y_\mathfrak o),\end{aligned}\end{eqnarray}where the subscripts ``l'' and ``r'' refer to the left and right direct summands of $\mathfrak o(\infty)\oplus\mathfrak o(\infty)$.
Since $D_\mathfrak o^\mathfrak{sp}\operatorname{U}(\mathfrak o(\infty))=\operatorname{U}(\mathfrak{sp}(\infty))$ according to Theorem~\ref{Tosp}, and $D_\mathfrak o^\mathfrak{sp}\Delta_\mathfrak o^U=\Delta_\mathfrak{sp}^U
$, formula~(\ref{Fl}) implies that $$D_\mathfrak o^\mathfrak{sp}(\operatorname{Ann}_{\operatorname{U}(\mathfrak o(\infty))}(X_\mathfrak o\otimes Y_\mathfrak o))=\operatorname{Ann}_{\operatorname{U}(\mathfrak{sp}(\infty))}(X_\mathfrak{sp}\otimes Y_{\mathfrak{sp}}).$$
\end{proof}


\end{document}